\documentclass[11pt]{amsart}
\usepackage{latexsym}
\usepackage{amsmath, amssymb, amsfonts, amsxtra, amsthm}
\usepackage{appendix}
\usepackage{bbm}
\usepackage{verbatim}
\usepackage{enumerate}
\usepackage[usenames]{color}
\begin{document}

\title{A variation norm Carleson theorem}

\begin{abstract}  We strengthen the Carleson-Hunt theorem by
proving $L^p$ estimates for the  $r$-variation of the partial
sum operators for Fourier series and integrals, for $r>\max\{p',2\}$.
Four appendices are concerned
with transference, a variation norm Menshov-Paley-Zygmund theorem,
and  applications to  nonlinear Fourier
transforms and  ergodic theory.
\end{abstract}

\author
[R. Oberlin, A. Seeger, T. Tao, C. Thiele, J. Wright]
{Richard Oberlin \ \ \ Andreas Seeger \ \ \ Terence Tao
\\ Christoph Thiele \ \ \  \  \ \ James Wright}

\address{R. Oberlin, Department of Mathematics,
Louisiana State University,
Baton Rouge, LA 70803-4918, USA}
\email{richard.oberlin@gmail.com}

\address{A. Seeger, Department of Mathematics, University of Wisconsin,
480 Lincoln Drive,
Madison, WI, 53706, USA}
\email{seeger@math.wisc.edu}

\address{T. Tao, Department of Mathematics,
UCLA, Los Angeles, CA 90095-1555, USA}
\email{tao@math.ucla.edu}

\address{C. Thiele, Department of Mathematics,
UCLA, Los Angeles, CA 90095-1555, USA}
\email{thiele@math.ucla.edu}

\address{J. Wright,
Maxwell Institute of Mathematical Sciences and the School of Mathematics,  University of
Edinburgh, JCMB, King's Buildings, Mayfield Road, Edinburgh EH9 3JZ,
Scotland}
\email{J.R.Wright@ed.ac.uk}


\thanks{R.O. partially supported by NSF VIGRE grant
DMS 0502315.
A.S. partially supported by NSF grant DMS 0652890.
T.T. partially supported by NSF Research Award
DMS-0649473, the NSF Waterman award and a grant from the
McArthur Foundation.
C.Th. partially supported by NSF grant DMS 0701302.
J.W. partially supported by an EPSRC grant.}

\date{\today}


\newcommand{\indic}[1]{\mathbbm{1}_{ {#1}  }}
\newcommand{\norm}[1]{ \left|  #1 \right| }
\newcommand{\Norm}[1]{ \left\|  #1 \right\| }
\newcommand{\ifof}{\Leftrightarrow}
\newcommand{\set}[1]{ \left\{ #1 \right\} }
\newcommand{\floor}[1]{\lfloor #1 \rfloor }

\def\cXtop{{\Xi^{\rm{top}}}}
\def\dyad{{\text{\rm dyad}}}
\def\intslash{\rlap{\kern  .32em $\mspace {.5mu}\backslash$ }\int}
\def\qsl{{\rlap{\kern  .32em $\mspace {.5mu}\backslash$ }\int_{Q_x}}}
\def\Re{\operatorname{Re\,}}
\def\Im{\operatorname{Im\,}}
\def\mx{{\max}}
\def\mn{{\min}}\def\vth{\vartheta}

\def\ene{{\mathcal E}}
\def\den{{\mu}}

\def\complex{{\mathbb C}}
\def\norm#1{{ \left|  #1 \right| }}
\def\Norm#1{{ \left\|  #1 \right\| }}
\def\set#1{{ \left\{ #1 \right\} }}
\def\floor#1{{\lfloor #1 \rfloor }}
\def\emph#1{{\it #1 }}
\def\diam{{\text{\rm diam}}}
\def\osc{{\text{\rm osc}}}
\def\ffB{\mathcal B}
\def\seq{\subseteq}
\def\Id{\text{\sl Id}}

\def\Ga{\Gamma}
\def\ga{\gamma}
\def\Th{\Theta}

\def\prd{{\text{\it prod}}}
\def\parab{{\text{\it parabolic}}}

\def\eg{{\it e.g. }}
\def\cf{{\it cf}}
\def\Rn{{\mathbb R^n}}
\def\Rd{{\mathbb R^d}}
\def\per{{\text{\rm per}}}
\def\sgn{{\text{\rm sign }}}
\def\rank{{\text{\rm rank }}}
\def\corank{{\text{\rm corank }}}
\def\coker{{\text{\rm Coker }}}
\def\loc{{\text{\rm loc}}}
\def\spec{{\text{\rm spec}}}

\def\comp{{\text{\rm comp}}}

\def\Coi{{C^\infty_0}}
\def\dist{{\text{\rm dist}}}
\def\diag{{\text{\rm diag}}}
\def\supp{{\text{\rm supp }}}
\def\rad{{\text{\rm rad}}}
\def\Lip{{\text{\rm Lip}}}
\def\inn#1#2{\langle#1,#2\rangle}
\def\biginn#1#2{\big\langle#1,#2\big\rangle}
\def\rta{\rightarrow}
\def\lta{\leftarrow}
\def\noi{\noindent}
\def\lcontr{\rfloor}
\def\lco#1#2{{#1}\lcontr{#2}}
\def\lcoi#1#2{\imath({#1}){#2}}
\def\rco#1#2{{#1}\rcontr{#2}}
\def\bin#1#2{{\pmatrix {#1}\\{#2}\endpmatrix}}
\def\meas{{\text{\rm meas}}}

\def\card{\text{\rm card}}
\def\lc{\lesssim}
\def\gc{\gtrsim}
\def\pv{\text{\rm p.v.}}

\def\alp{\alpha}             \def\Alp{\Alpha}
\def\bet{\beta}
\def\gam{\gamma}             \def\Gam{\Gamma}
\def\del{\delta}             \def\Del{\Delta}
\def\ep{\epsilon}
\def\zet{\zeta}
\def\tet{\theta}             \def\Tet{\Theta}
\def\iot{\iota}
\def\kap{\kappa}
\def\ka{\kappa}
\def\lam{\lambda}            \def\Lam{\Lambda}
\def\la{\lambda}             \def\La{\Lambda}
\def\sig{\sigma}             \def\Sig{\Sigma}
\def\si{\sigma}              \def\Si{\Sigma}
\def\vphi{\varphi}
\def\ome{\omega}             \def\Ome{\Omega}
\def\om{\omega}              \def\Om{\Omega}

\def\fA{{\mathfrak {A}}}
\def\fB{{\mathfrak {B}}}
\def\fC{{\mathfrak {C}}}
\def\fD{{\mathfrak {D}}}
\def\fE{{\mathfrak {E}}}
\def\fF{{\mathfrak {F}}}
\def\fG{{\mathfrak {G}}}
\def\fH{{\mathfrak {H}}}
\def\fI{{\mathfrak {I}}}
\def\fJ{{\mathfrak {J}}}
\def\fK{{\mathfrak {K}}}
\def\fL{{\mathfrak {L}}}
\def\fM{{\mathfrak {M}}}
\def\fN{{\mathfrak {N}}}
\def\fO{{\mathfrak {O}}}
\def\fP{{\mathfrak {P}}}
\def\fQ{{\mathfrak {Q}}}
\def\fR{{\mathfrak {R}}}
\def\fS{{\mathfrak {S}}}
\def\fT{{\mathfrak {T}}}
\def\fU{{\mathfrak {U}}}
\def\fV{{\mathfrak {V}}}
\def\fW{{\mathfrak {W}}}
\def\fX{{\mathfrak {X}}}
\def\fY{{\mathfrak {Y}}}
\def\fZ{{\mathfrak {Z}}}

\def\fa{{\mathfrak {a}}}
\def\fb{{\mathfrak {b}}}
\def\fc{{\mathfrak {c}}}
\def\fd{{\mathfrak {d}}}
\def\fe{{\mathfrak {e}}}
\def\ff{{\mathfrak {f}}}
\def\fg{{\mathfrak {g}}}
\def\fh{{\mathfrak {h}}}
\def\fj{{\mathfrak {j}}}
\def\fk{{\mathfrak {k}}}
\def\fl{{\mathfrak {l}}}
\def\fm{{\mathfrak {m}}}
\def\fn{{\mathfrak {n}}}
\def\fo{{\mathfrak {o}}}
\def\fp{{\mathfrak {p}}}
\def\fq{{\mathfrak {q}}}
\def\fr{{\mathfrak {r}}}
\def\fs{{\mathfrak {s}}}
\def\ft{{\mathfrak {t}}}
\def\fu{{\mathfrak {u}}}
\def\fv{{\mathfrak {v}}}
\def\fw{{\mathfrak {w}}}
\def\fx{{\mathfrak {x}}}
\def\fy{{\mathfrak {y}}}
\def\fz{{\mathfrak {z}}}

\def\bbA{{\mathbb {A}}}
\def\bbB{{\mathbb {B}}}
\def\bbC{{\mathbb {C}}}
\def\bbD{{\mathbb {D}}}
\def\bbE{{\mathbb {E}}}
\def\bbF{{\mathbb {F}}}
\def\bbG{{\mathbb {G}}}
\def\bbH{{\mathbb {H}}}
\def\bbI{{\mathbb {I}}}
\def\bbJ{{\mathbb {J}}}
\def\bbK{{\mathbb {K}}}
\def\bbL{{\mathbb {L}}}
\def\bbM{{\mathbb {M}}}
\def\bbN{{\mathbb {N}}}
\def\bbO{{\mathbb {O}}}
\def\bbP{{\mathbb {P}}}
\def\bbQ{{\mathbb {Q}}}
\def\bbR{{\mathbb {R}}}
\def\bbS{{\mathbb {S}}}
\def\bbT{{\mathbb {T}}}
\def\bbU{{\mathbb {U}}}
\def\bbV{{\mathbb {V}}}
\def\bbW{{\mathbb {W}}}
\def\bbX{{\mathbb {X}}}
\def\bbY{{\mathbb {Y}}}
\def\bbZ{{\mathbb {Z}}}

\def\cA{{\mathcal {A}}}
\def\cB{{\mathcal {B}}}
\def\cC{{\mathcal {C}}}
\def\cD{{\mathcal {D}}}
\def\cE{{\mathcal {E}}}
\def\cF{{\mathcal {F}}}
\def\cG{{\mathcal {G}}}
\def\cH{{\mathcal {H}}}
\def\cI{{\mathcal {I}}}
\def\cJ{{\mathcal {J}}}
\def\cK{{\mathcal {K}}}
\def\cL{{\mathcal {L}}}
\def\cM{{\mathcal {M}}}
\def\cN{{\mathcal {N}}}
\def\cO{{\mathcal {O}}}
\def\cP{{\mathcal {P}}}
\def\cQ{{\mathcal {Q}}}
\def\cR{{\mathcal {R}}}
\def\cS{{\mathcal {S}}}
\def\cT{{\mathcal {T}}}
\def\cU{{\mathcal {U}}}
\def\cV{{\mathcal {V}}}
\def\cW{{\mathcal {W}}}
\def\cX{{\mathcal {X}}}
\def\cY{{\mathcal {Y}}}
\def\cZ{{\mathcal {Z}}}

\def\tA{{\widetilde{A}}}
\def\tB{{\widetilde{B}}}
\def\tC{{\widetilde{C}}}
\def\tD{{\widetilde{D}}}
\def\tE{{\widetilde{E}}}
\def\tF{{\widetilde{F}}}
\def\tG{{\widetilde{G}}}
\def\tH{{\widetilde{H}}}
\def\tI{{\widetilde{I}}}
\def\tJ{{\widetilde{J}}}
\def\tK{{\widetilde{K}}}
\def\tL{{\widetilde{L}}}
\def\tM{{\widetilde{M}}}
\def\tN{{\widetilde{N}}}
\def\tO{{\widetilde{O}}}
\def\tP{{\widetilde{P}}}
\def\tQ{{\widetilde{Q}}}
\def\tR{{\widetilde{R}}}
\def\tS{{\widetilde{S}}}
\def\tT{{\widetilde{T}}}
\def\tU{{\widetilde{U}}}
\def\tV{{\widetilde{V}}}
\def\tW{{\widetilde{W}}}
\def\tX{{\widetilde{X}}}
\def\tY{{\widetilde{Y}}}
\def\tZ{{\widetilde{Z}}}

\def\tcA{{\widetilde{\mathcal {A}}}}
\def\tcB{{\widetilde{\mathcal {B}}}}
\def\tcC{{\widetilde{\mathcal {C}}}}
\def\tcD{{\widetilde{\mathcal {D}}}}
\def\tcE{{\widetilde{\mathcal {E}}}}
\def\tcF{{\widetilde{\mathcal {F}}}}
\def\tcG{{\widetilde{\mathcal {G}}}}
\def\tcH{{\widetilde{\mathcal {H}}}}
\def\tcI{{\widetilde{\mathcal {I}}}}
\def\tcJ{{\widetilde{\mathcal {J}}}}
\def\tcK{{\widetilde{\mathcal {K}}}}
\def\tcL{{\widetilde{\mathcal {L}}}}
\def\tcM{{\widetilde{\mathcal {M}}}}
\def\tcN{{\widetilde{\mathcal {N}}}}
\def\tcO{{\widetilde{\mathcal {O}}}}
\def\tcP{{\widetilde{\mathcal {P}}}}
\def\tcQ{{\widetilde{\mathcal {Q}}}}
\def\tcR{{\widetilde{\mathcal {R}}}}
\def\tcS{{\widetilde{\mathcal {S}}}}
\def\tcT{{\widetilde{\mathcal {T}}}}
\def\tcU{{\widetilde{\mathcal {U}}}}
\def\tcV{{\widetilde{\mathcal {V}}}}
\def\tcW{{\widetilde{\mathcal {W}}}}
\def\tcX{{\widetilde{\mathcal {X}}}}
\def\tcY{{\widetilde{\mathcal {Y}}}}
\def\tcZ{{\widetilde{\mathcal {Z}}}}

\def\tfA{{\widetilde{\mathfrak {A}}}}
\def\tfB{{\widetilde{\mathfrak {B}}}}
\def\tfC{{\widetilde{\mathfrak {C}}}}
\def\tfD{{\widetilde{\mathfrak {D}}}}
\def\tfE{{\widetilde{\mathfrak {E}}}}
\def\tfF{{\widetilde{\mathfrak {F}}}}
\def\tfG{{\widetilde{\mathfrak {G}}}}
\def\tfH{{\widetilde{\mathfrak {H}}}}
\def\tfI{{\widetilde{\mathfrak {I}}}}
\def\tfJ{{\widetilde{\mathfrak {J}}}}
\def\tfK{{\widetilde{\mathfrak {K}}}}
\def\tfL{{\widetilde{\mathfrak {L}}}}
\def\tfM{{\widetilde{\mathfrak {M}}}}
\def\tfN{{\widetilde{\mathfrak {N}}}}
\def\tfO{{\widetilde{\mathfrak {O}}}}
\def\tfP{{\widetilde{\mathfrak {P}}}}
\def\tfQ{{\widetilde{\mathfrak {Q}}}}
\def\tfR{{\widetilde{\mathfrak {R}}}}
\def\tfS{{\widetilde{\mathfrak {S}}}}
\def\tfT{{\widetilde{\mathfrak {T}}}}
\def\tfU{{\widetilde{\mathfrak {U}}}}
\def\tfV{{\widetilde{\mathfrak {V}}}}
\def\tfW{{\widetilde{\mathfrak {W}}}}
\def\tfX{{\widetilde{\mathfrak {X}}}}
\def\tfY{{\widetilde{\mathfrak {Y}}}}
\def\tfZ{{\widetilde{\mathfrak {Z}}}}

\def\Atil{{\widetilde A}}          \def\atil{{\tilde a}}
\def\Btil{{\widetilde B}}          \def\btil{{\tilde b}}
\def\Ctil{{\widetilde C}}          \def\ctil{{\tilde c}}
\def\Dtil{{\widetilde D}}          \def\dtil{{\tilde d}}
\def\Etil{{\widetilde E}}          \def\etil{{\tilde e}}
\def\Ftil{{\widetilde F}}          \def\ftil{{\tilde f}}
\def\Gtil{{\widetilde G}}          \def\gtil{{\tilde g}}
\def\Htil{{\widetilde H}}          \def\htil{{\tilde h}}
\def\Itil{{\widetilde I}}          \def\itil{{\tilde i}}
\def\Jtil{{\widetilde J}}          \def\jtil{{\tilde j}}
\def\Ktil{{\widetilde K}}          \def\ktil{{\tilde k}}
\def\Ltil{{\widetilde L}}          \def\ltil{{\tilde l}}
\def\Mtil{{\widetilde M}}          \def\mtil{{\tilde m}}
\def\Ntil{{\widetilde N}}          \def\ntil{{\tilde n}}
\def\Otil{{\widetilde O}}          \def\otil{{\tilde o}}
\def\Ptil{{\widetilde P}}          \def\ptil{{\tilde p}}
\def\Qtil{{\widetilde Q}}          \def\qtil{{\tilde q}}
\def\Rtil{{\widetilde R}}          \def\rtil{{\tilde r}}
\def\Stil{{\widetilde S}}          \def\stil{{\tilde s}}
\def\Ttil{{\widetilde T}}          \def\ttil{{\tilde t}}
\def\Util{{\widetilde U}}          \def\util{{\tilde u}}
\def\Vtil{{\widetilde V}}          \def\vtil{{\tilde v}}
\def\Wtil{{\widetilde W}}          \def\wtil{{\tilde w}}
\def\Xtil{{\widetilde X}}          \def\xtil{{\tilde x}}
\def\Ytil{{\widetilde Y}}          \def\ytil{{\tilde y}}
\def\Ztil{{\widetilde Z}}          \def\ztil{{\tilde z}}


\def\ahat{{\hat a}}          \def\Ahat{{\widehat A}}
\def\bhat{{\hat b}}          \def\Bhat{{\widehat B}}
\def\chat{{\hat c}}          \def\Chat{{\widehat C}}
\def\dhat{{\hat d}}          \def\Dhat{{\widehat D}}
\def\ehat{{\hat e}}          \def\Ehat{{\widehat E}}
\def\fhat{{\hat f}}          \def\Fhat{{\widehat F}}
\def\ghat{{\hat g}}          \def\Ghat{{\widehat G}}
\def\hhat{{\hat h}}          \def\Hhat{{\widehat H}}
\def\ihat{{\hat i}}          \def\Ihat{{\widehat I}}
\def\jhat{{\hat j}}          \def\Jhat{{\widehat J}}
\def\khat{{\hat k}}          \def\Khat{{\widehat K}}
\def\lhat{{\hat l}}          \def\Lhat{{\widehat L}}
\def\mhat{{\hat m}}          \def\Mhat{{\widehat M}}
\def\nhat{{\hat n}}          \def\Nhat{{\widehat N}}
\def\ohat{{\hat o}}          \def\Ohat{{\widehat O}}
\def\phat{{\hat p}}          \def\Phat{{\widehat P}}
\def\qhat{{\hat q}}          \def\Qhat{{\widehat Q}}
\def\rhat{{\hat r}}          \def\Rhat{{\widehat R}}
\def\shat{{\hat s}}          \def\Shat{{\widehat S}}
\def\that{{\hat t}}          \def\That{{\widehat T}}
\def\uhat{{\hat u}}          \def\Uhat{{\widehat U}}
\def\vhat{{\hat v}}          \def\Vhat{{\widehat V}}
\def\what{{\hat w}}          \def\What{{\widehat W}}
\def\xhat{{\hat x}}          \def\Xhat{{\widehat X}}
\def\yhat{{\hat y}}          \def\Yhat{{\widehat Y}}
\def\zhat{{\hat z}}          \def\Zhat{{\widehat Z}}

\def\tg{{\widetilde g}}

\def\Be#1{\begin{equation}\label{#1}}
\def\Ee{\end{equation}}
\def\endal{\end{align}}
\def\bas{\begin{align*}}
\def\eas{\end{align*}}
\def\bi{\begin{itemize}}
\def\ei{\end{itemize}}

\def\dim{{\hbox{\roman dim}}}
\def\dimsup{{\roman \overline{dim}}}
\def\dir{{\hbox{\roman dir}}}
\def\rp{{r^\prime}}
\def\rpt{{\tilde r^\prime}}
\def \endprf{\hfill  {\vrule height6pt width6pt depth0pt}\medskip}
\def\emph#1{{\it #1}}
\def\textbf#1{{\bf #1}}

\parskip = 12 pt

\newcommand {\rea}{\mathbb{R}}
\newcommand {\R}{\rea}
\newcommand {\eps}{\epsilon}
\newcommand {\ma}{\mathcal{M}}
\newcommand {\squaref}{\mathbb{S}}
\newcommand {\intervals}{\mathcal{I}}
\newcommand {\proj}{\mathrm{proj}}
\newcommand {\J}{\mathbf{J}}
\newcommand {\D}{\mathcal{D}}
\newcommand {\Z}{\mathbb{Z}}
\newcommand {\N}{\mathbb{N}}
\newcommand {\T}{\mathbf{T}}
\newcommand {\CS}{\mathcal{S}}
\newcommand {\CC}{\mathcal{C}}
\newcommand {\energy}{\mathop{\mathrm{energy}}}
\newcommand {\density}{\mathop{\mathrm{density}}}
\newcommand {\densityu}{\mathop{\mathrm{density}_u}}
\newcommand {\densityl}{\mathop{\mathrm{density}_l}}
\newcommand {\densityov}{\mathop{\mathrm{density}_{ov}}}
\newcommand {\densityla}{\mathop{\mathrm{density}_{la}}}
\newcommand {\cent}{\mathop{\mathrm{center}}}
\newcommand {\expect}{\mathbb{E}}
\newcommand {\diff}{\mathbb{D}}
\newcommand {\dis}{\mathrm{dist}}
\newcommand {\real}{\mathop{\mathrm{Re}}}
\newcommand {\imaginary}{\mathop{\mathrm{Im}}}

\def \P{\mathbf{P}}
\def\<{\langle}
\def\>{\rangle}

\newtheorem{theorem}{Theorem}[section]
\newtheorem*{definition}{Definition}
\newtheorem{conjecture}[theorem]{Conjecture}
\newtheorem{proposition}[theorem]{Proposition}
\newtheorem{claim}[theorem]{Claim}
\newtheorem{lemma}[theorem]{Lemma}
\newtheorem{corollary}[theorem]{Corollary}
\newtheorem{observation}[theorem]{Observation}

\theoremstyle{remark}
   \newtheorem{remark}[theorem]{Remark}
\newtheorem*{remarknotag}{Remark}

\topmargin -.25in

\maketitle

\section{Introduction}

For an  integrable function $f$
 on the circle  group $\bbT= \bbR/\bbZ$, and $k\in \bbZ$  we denote by
$\widehat f_k=\,\int_0^1 f(y)e^{-2\pi i ky} \,dy$
the
Fourier coefficients  and
consider  the partial sum operators $S_n$ for the Fourier
series,
\Be{partialsum} S_n f(x) \equiv S[f](n,x)= \sum_{k=-n}^n
\widehat f_k \,e^{2\pi i kx};\Ee
here
$n\in \bbN_0=\{0,1,2,\dots\}$.
The  celebrated
theorem by Carleson \cite{carleson66cgp} states that if $f$ is square
 integrable then
$S_n f$ converges to $f$ almost everywhere. Hunt
\cite{hunt68cfs} extended this result
 to $L^p(\bbT)$ functions, for $1<p<\infty$,  and  proved the inequality
\Be{carlesonhunt} \big\|\sup_n|S_n f|\big\|_{L^p(\bbT)} \le
C \,\|f\|_{L^p(\bbT)}
\Ee
for all $f\in L^p(\bbT)$;
see also  \cite{fefferman73pcf}, \cite{lacey00pbc}, and
 \cite{grafakos04bmd}
for other proofs  of this fact.

The purpose of this paper is to
strengthen the Carleson-Hunt  result for $L^p$ functions,
$1<p<\infty$, and
show that, for
 $r>\max\{2,p'\}$,  the (strong) $r$-variation of the sequence $\{S_n f(x)\}_{n\in
  \bbN_0}$ is finite for almost every $x\in [0,1]$.
This can be interpreted as a
 statement about the rate of convergence.
To fix notation,
we consider real or complex valued  sequences  $\{ a_n\}_{n\in \bbN_0}$
 and define their
$r$-variation to be
\begin{equation}\label{definitionofvariation}
\|a\|_{V^r}
=\sup_K  \sup_{n_0 < \cdots < n_K} \Bigl(
\sum_{\ell = 1}^{K} |a_{n_{\ell}}  - a_{n_{\ell-1}} |^r
\Bigr)^{1/r}
\end{equation}
where the sup is taken over all $K$ and then over all
increasing sequences of
nonnegative integers $n_0 < \cdots < n_K$. Note that the variation norms
are monotone decreasing in the parameter $r$.
Next, for a sequence $F=\{F_n\}$ of Lebesgue measurable functions
one defines  the  $r$-variation of $F$ at $x$, sometimes denoted by
$\cV^rF(x)$ as the $V^r$ norm of the sequence $\{F_n(x)\}$.
We denote
the  $r$-variation of the sequence $F_n =S_nf$
 by
$\cV^r S[f]$.
The variation norms and the  $r$-variation operator
 can be  defined in a similar fashion, if the index set $\bbN_0$ is
replaced by another subset
of $\bbR$
(often  $\bbR^+$ or $\bbR$ itself).

Let $r':=r/(r-1)$, the conjugate exponent of $r$.
\begin{theorem} \label{maintheoremper}
Suppose $r > 2$ and $r' < p < \infty.$ Then, for every $f\in L^p(\bbT)$,
\begin{equation} \label{maintheoremperinequality}
\big\|S [f]\big\|_{L^p(V^r)} \leq C_{p,r} \|f\|_{L^p}.
\end{equation}
At the endpoint $p=r'$ a restricted weak type result holds;
namely, for any $f\in L^{r',1}(\bbT)$ the function
$\cV^r S[f]$ belongs to $L^{r',\infty}(\bbT)$.
\end{theorem}
It is immediate that \eqref{maintheoremperinequality} for $r<\infty$
implies a quantitative form of almost everywhere convergence of Fourier series, 
improving over the standard qualitative result utilizing the weaker
$r=\infty$ inequality and convergence on a dense subclass of functions.

As will be
discussed in Section \ref{counterexamplesection}, the conditions
on the exponents in \eqref{maintheoremperinequality}
are sharp. Moreover, in the endpoint case $p=r'$ the Lorentz space
$L^{r',\infty}$ cannot be replaced by a smaller Lorentz space.

%

By standard transference arguments (see Appendix \ref{transference})
Theorem \ref{maintheoremper} is implied by
a result on the  partial (inverse) Fourier integral of a Schwartz
function
$f$ on $\rea$
is defined as
$$\CS[f](\xi,x) = \int_{-\infty}^{\xi} \widehat{f}(\eta)
e^{2\pi i \eta x}
\ d\eta
$$ where $\widehat{f}(\eta)= \int f(y)
e^{-2\pi i y\eta}\, dy$
 defines the Fourier transform of $f$.

\begin{theorem} \label{maintheorem}
Suppose $r > 2$.
Then $\cS$ extends to a bounded operator $\cS: L^p\to L^p(V^r)$ for
$r'<p<\infty$.
Moreover $\cS$ maps $L^{r',1}$ boundedly to $L^{r',\infty}(V^r)$.
\end{theorem}

Note that if in the above definition of the mixed $L^p(V^r)$ spaces
we
interchange the order between integration in the $x$ variable and taking the supremum over the choices of $K$ and the points $\xi_0$ to $\xi_K$
so that these choices become independent of the variable $x$, then the
estimates corresponding to Theorem
\ref{maintheorem} are weaker; they  follow  from  a square function
 inequality of Rubio de Francia \cite{rubio85lp} for $p\ge 2$,
see also  \cite{quek} for a related endpoint
result
for $p<2$,    and
\cite{lacey07rf} for a proof of Rubio de Francia's inequality
which is closer to the methods of this
paper.

While the concept of $r$-variation norm is at least as
old as Wiener's 1920s paper on quadratic
variation \cite{wiener24qv},  variational estimates
have been pioneered by D. L\'epingle (\cite{lepingle76lvd}) who proved them for martingales.
Simple proofs of  L\'epingle's result based on jump inequalities
have been given by Pisier and Xu \cite{pisier-xu} and
by Bourgain \cite{bourgain89pet},
and applications to other families of
operators in harmonic analysis such as families of
averages and singular integrals have  been considered  in
\cite{bourgain89pet}, and the subsequent papers 
\cite{jkrw}, \cite{cjrw-hilbert},
\cite{jones08svj} (\cf. the bibliography of \cite{jones08svj} for more
references).
Bourgain \cite{bourgain89pet} used variation norm  estimates
(or related oscillation estimates which are intermediate in difficulty
between maximal and variation norm estimates)
to prove pointwise convergence results
without previous knowledge
that pointwise convergence holds for a dense subclass of functions.
Such dense subclasses of functions,
while usually available in the setting of analysis on Euclidean space,
are less abundant in the
ergodic theory setting. In Appendix \ref{ergodicappendix}
we demonstrate the use of Theorem \ref{maintheorem}
in the setting of Wiener-Wintner type theorems as developed in
\cite{lacey08wwt}. We note that the  Carleson-Hunt theorem has previously been generalized by using
other norms in place of the variation norm, see for example the use of
oscillation norms in \cite{lacey08wwt}, and the $M^*_2$
 norms in   \cite{demeter08bdr}, \cite{demeter08wmc}.

We are also  motivated by the fact that variation norms are in certain
situations more stable under nonlinear perturbation than supremum norms.
For example one can deduce bounds for certain $r$-variational lengths of curves in Lie groups from the corresponding lengths of the ``trace'' of the curves in
the corresponding Lie algebras, see Appendix \ref{vliegroupsection} for definitions and details. What we have in mind is proving Carleson type theorems
for nonlinear perturbations of the Fourier transform as discussed in
\cite{muscalu03ctc}, \cite{muscalu03cme}. Unfortunately
the naive approach fails and the ultimate goal remains unattained
since we only know the correlation between lengths of the trace and
the original curve for $r < 2$, while the variational Carleson theorem
only holds for $r > 2.$ Nonetheless, this method allows one to see
that a variational version of the Christ-Kiselev theorem
\cite{christ01wab} follows from a variational Menshov-Paley-Zygmund
theorem which we prove in Appendix \ref{vmpzsection}. The variational
Carleson inequality  can be viewed as an
endpoint in this theory.

Our proof of Theorem \ref{maintheorem} will follow the method of \cite{lacey00pbc} as refined in \cite{grafakos04bmd}. Naturally one has to invoke variation norm results in the setting of individual trees, 
which is achieved by adapting D. L\'epingle's result (\cite{lepingle76lvd}) to the
setting of a tree. The authors initially had a proof of the case $p>2$ and $r>p$ of Theorem 
\ref{maintheorem} more akin to \cite{lacey00pbc}, while improvement to $r>2$ for such $p$
provided a stumbling block. This stumbling block was removed by better accounting for trees of given 
energy, as described in the remarks leading to Proposition \ref{energyincrementprop}.
In Section \ref{mosection} we reduce the problem to that of bounding certain model operators which map $f$ to linear combinations of wave-packets associated to collections of multitiles. In Section \ref{treeestimatessection} we bound the model operators when the collection of multitiles is of a certain type called a tree; this bound is in terms of two quantities, energy and density, which are associated to the tree. These quantities are defined in Section \ref{energydensitysection} and an algorithm is given to decompose an arbitrary collection of multitiles into a union of trees with controlled energy and
density. Section \ref{auxiliarysection} contains two auxiliary estimates.
All these ingredients are combined to complete the proof
in Section \ref{mainargumentsection}.

The authors would like to thank the anonymous referee for valuable comments
on the submitted draft of the paper.

\medskip

{\it Some notation.} For two quantities $A$ and $B$ let $A\lc B$ denote the statement that
$A\le CB$ for some constant $C$ (possibly depending on the parameters $p$ and $r$). The Lebesgue measure of a set $E$ is denoted either
 by $|E|$ or by $\meas(E)$. The indicator function of $E$ is
 denoted by  $\indic{E}$. For a subset $E\subset \bbR$ and $a\in \bbR$ we set
$a+E=E+a=
\{x:x-a\in E\}$.
If $I$ is a finite    interval $I$ with center $c(I)$   we denote by
$CI$
 the $C$ dilate of $I$ {\it with respect to
its center}, i.e. the  set of all $x$ for which $c(I)+
 \frac{x-c(I)}{C}\in I$.

\section{Optimality of the exponents} \label{counterexamplesection}
Since 
 Theorem 
 \ref{maintheorem} implies
 Theorem \ref{maintheoremper} (\cf.  Appendix \ref{transference})
we have to discuss the
 optimality only for the Fourier series case.
The necessity of the condition $r>2$  follows from a 
corresponding result for the  Cesaro means; its proof   by Jones and Wang
 \cite{jones04vif} was 
based on a probabilistic result of  Qian  \cite{qian98pvp}.

We show the  necessity of the condition $p > r'$ in Theorem 
\ref{maintheoremper}.
Let $$D_n(x)= \sum_{k=-n}^n e^{2\pi i kx}=\frac{\sin((2n+1)\pi
  x)}{\sin(\pi x)},$$ the Dirichlet kernel,
and let $f^N$ be the de la Vall\'ee-Poussin kernel which is defined 
by 
$f^{N}=2K_{2N+1}-K_N$ via the Fej\'er kernel
$K_N=(N+1)^{-1}\sum_{j=0}^N D_j$.  
Then $[\widehat {f^N}]_k=1$ for $|k|\le N+1$  and thus 
$S_n f^N = D_n$ for $|n|\le N+1$. We have  $\|f^N\|_{L^1(\bbT)}=O(1)$,
and  $\|f^N\|_\infty= O(N)$ and therefore
$\|f^N\|_{L^{p,q}(\bbT)}= O(N^{1-1/p})$.

Let $N\gg 10^3$ and $ 8N^{-1} \le x\le  1/8$. Let $K=K(x)$ be the largest
integer $< Nx$.
Then for $0\le k< 2K(x)$ there are integers  $n_k(x)\le N$ so that
$(2n_k(x)+1)x \in (\frac14+k, \frac 34+k)$, in particular
$n_k(x)<n_{k+1}(x)$
for $k<2K(x)-1$.
Observe
$\sin((2n_{2j}(x)+1)\pi
  x) >\sqrt 2/2$ and 
$\sin((2n_{2j+1}(x)+1)\pi
  x) <-\sqrt 2/2$ for $0\le j\le K(x)-1$. 
This gives 
\begin{multline*}
\Big(\sum_{j=0}^{K(x)-1}\big| S_{n_{2j+1}(x)}f^N - S_{n_{2j}(x)}f^N
\big|^r\Big)^{1/r} 
=
\Big(\sum_{j=0}^{K(x)-1}\big| D_{n_{2j+1}(x)} - D_{n_{2j}(x)} \big|^r\Big)^{1/r} 
\\\ge \frac{K(x)^{1/r}\sqrt 2}{\sin(\pi x)} \ge  c N^{1/r} x^{1/r-1},
\end{multline*}
and this implies for large $N$
\[\frac{\|{\cV}^r(S f^N)\|_{L^{p,s}}}{\|f^N\|_{L^{p,1}}}\ge c_{p,s} 
\begin{cases} 
N^{\frac 1p -\frac 1{r'}} &\text{  if $p<r'$, }
\\(\log N)^{1/s} &\text{  if $p=r'$.}
\end{cases}
\]
Thus 
the $L^p\to L^p(V^r)$ boundedness does not hold  for $p<r'$; moreover 
the $L^{r',1}\to L^{r',s}(V^{r})$ boundedness does not hold for
$s<\infty$.

\section{The model operators}\label{mosection}
We shall show  in appendix \ref{transference} how to deduce Theorem
\ref{maintheoremper} from Theorem \ref{maintheorem}.
To start the proof of the main  Theorem \ref{maintheorem},
we describe some reductions to  model operators involving wave packet
decompositions.

First, by interpolation it suffices to prove for $p\ge r'$ the
restricted weak type $L^{p,1}\to L^{p,\infty}(V_r)$ bound. Next, by
  the monotone  convergence theorem it suffices to estimate
$L^{p,\infty}(V_r)$ on finite $x$-intervals $[-A,A]$, with constant
  independent of $A$.
By another application of the monotone convergence theorem it suffices, for any fixed $K$,  to prove
the
$L^{p,1}\to L^{p,\infty}([-A,A])$ bound for
\Be{fixedK}
\sup_{\xi_0\le \cdots\le \xi_{K} }\Big(\sum_{\ell=1}^K\big| \cS f(\xi_\ell,x)-\cS f(\xi_{\ell-1},x)\big|^r\Big)^{1/r}
\Ee
where the sup is taken over all $(\xi_0,\dots,\xi_K)$ with $\xi_{\ell-1}\le \xi_\ell$ for $\ell=1,\dots, K$.
Moreover, by the density of Schwartz functions in $L^{p,1}$ it
suffices to prove a uniform estimate  for all
Schwartz functions.
Note that for any Schwartz function $f$ the
expression $\cS[f](\xi,x)$ depends continuously  on
$(\xi,x)$.
Therefore it suffices to bound the expression analogous to \eqref{fixedK} where
we impose the strict inequality $\xi_{\ell-1}< \xi_\ell$ for $\ell=1,\dots, K$.
Moreover, by the continuity it suffices for each finite set $\Xi\subset \bbR$ to prove bounds for this expression under the assumption that
the $\xi_\ell$ belong to $\Xi$, and we may also assume  that
$\Xi$ does not contain
any numbers of the form $n2^m$ with
$m, n\in \bbZ$ ({\it i.e.} no endpoints of dyadic intervals).

We may now linearize the variation norm. Fix $K\in \bbN$,
measurable real valued functions $\xi_{0}(x) < \ldots < \xi_{K}(x)$,
with values in $\Xi$ and measurable complex valued functions $a_1(x)$,
$\ldots$, $a_{K}(x)$ satisfying
$$|a_1(x)|^{r'} + \ldots + |a_{K}(x)|^{r'} = 1.$$
Let
\[
\CS'[f](x) = \sum_{k = 1}^{K}
\big(\CS[f](\xi_{k}(x),x) - \CS[f](\xi_{k-1}(x),x)\big)\,a_{k}(x)\, .
\]
Theorem \ref{maintheorem} will now follow  from
the estimate
\begin{equation} \label{linearizedtheorem}
\big\|\CS'[f]\big\|_{L^{p,\infty}(\rea)} \leq C \|f\|_{L^{p,1}(\rea)}
\end{equation}
where $C$ is independent of $K$, $\Xi$  and the linearizing functions,
 and where $f$ is any Schwartz function. Finally, in order to prove
\eqref{linearizedtheorem} for any fixed $\Xi$
we may assume that  $\widehat f$ has compact support in $\bbR\setminus
 \Xi$ since the space of Schwartz functions with this
property  is dense in
 $L^{p,1}$, $1<p<\infty$.

Let $\D = \{[2^k m, 2^{k}(m+1)) : m,k \in \Z\}$ be the set of dyadic intervals. A {\it tile} will be any rectangle $I \times \omega$ where $I,\omega$ are dyadic intervals, and $|I||\omega| = 1/2.$
We will write $\CS'$ as the sum of wave packets adapted to tiles, and then decompose the operator into a finite sum of model operators by sorting the wave packets into a finite number of classes.
For each $k$,
\[
\CS[f](\xi_{k},x) - \CS[f](\xi_{k-1},x) = \int \indic{(\xi_{k-1},\xi_{k})} (\xi) \widehat{f}(\xi) e^{2\pi i \xi x}\ d\xi.
\]
To suitably express the difference above as a sum of wave packets, we will first need to construct a partition of $\indic{(\xi_{k-1},\xi_{k})}$ adapted to certain dyadic intervals. The fact that $(\xi_{k-1},\xi_{k})$ has two boundary points instead of the one from $(-\infty,\xi_k)$ will necessitate a slightly more involved discretization argument than that in \cite{lacey00pbc}.

For any $\xi < \xi'$, let $\J_{\xi,\xi'}$ be the set of maximal dyadic
intervals $J$ such that $J \subset (\xi,\xi')$ and
$\dis(J,\xi) \geq |J|$,
$\dis(J,\xi') \geq |J|.$
Let $\nu$ be a $C^\infty$ function from $\rea$ to $[0,1]$ which vanishes on $(-\infty,-10^{-2}]$, is identically equal to $1$ on $[10^{-2},\infty)$, and so
that  $\nu'(x)\ge 0$ for $-10^{-2}<x< 10^{-2}$. Given an interval $J = [a,b)$, and $i \in \{-1,0,1\}$, define
\[
\varphi_{J,i}(\xi)
= \nu\left(\frac{\xi - a}{2^i(b-a)}\right) - \nu\left(\frac{\xi-b}{b-a}\right).
\]
Thus if $c(J)=\frac{a+b}2$, the center of $J$, then
\Be{nu-functions}\varphi_{J,i}(\xi)= \nu_i(\tfrac{\xi-c(J)}{|J|})
\text{ where } \nu_i(\eta)= \nu(2^{-i}(\eta+\tfrac 12))-\nu(\eta-\tfrac 12)\,,
\Ee
and we notice  for $i \in \{-1,0,1\}$  both
$\nu_i$ and $\sqrt {\nu_i}$ are $C^\infty$ functions  supported in
$[-\frac {13}{25}, \frac{13}{25}]$ (more precisely in $[-\frac {13}{25},
\frac{51}{100}]$).
 Hence
$\varphi_{J,i}$ is supported on a $\frac{26}{25}$-dilate of $J$ with
respect to its center.

For each  $J \in \J_{\xi,\xi'}$, one may check that there is a unique
interval $J' \in \J_{\xi,\xi'}$ which lies strictly to the left of $J$
and satisfies $\dis(J',J) = 0,$ and one may check that $J'$ has
size $|J|/2,|J|,$ or $2|J|$. We define $\varphi_J = \varphi_{J,i(J)}$ where $i(J)$ is chosen so that $|J'|=2^{i(J)}|J|$. Then
\Be{partitionofunity}
\indic{(\xi,\xi')}(\eta)\widehat f(\eta) = \sum_{J \in \J_{\xi,\xi'}}
\varphi_{J}(\eta)
\widehat f(\eta).
\Ee
Since we assume that $\widehat  f$ is compactly supported in $\bbR\setminus \Xi$ we see that that for every pair
$\xi<\xi'$ with $\xi, \xi'\in \Xi$ only a finite number of dyadic intervals
$J\in \bold J_{\xi,\xi'}$ are relevant in \eqref{partitionofunity}.

We now write each multiplier $\varphi_J$ as the sum of wave packets. For every tile $P = I \times J$, define $$\phi_{P}(x) = |I|^{1/2}\,
\cF^{-1}[\sqrt{\varphi_{J}}](x - c(I))$$ where $c(I)$ is the center of  $I$ and $\cF^{-1}$ denotes the inverse Fourier transform.
For each $J$, we then have
\Be{expansion}
\sum_{|I| = 1/(2|J|)}\<f,\phi_{I \times J}\>
\widehat{\phi_{I \times J}} = \widehat{f}\varphi_{J}.
\Ee
To see this we use a Fourier series expansion
(\cf.  \cite{thiele06wpa}). We  first observe that
$\widehat{\phi_P}(\xi) =\sqrt{|I|} \sqrt{\varphi_J(\xi)}
 e^{-2\pi i c(I)\xi}$
and use $\<f,\phi_{P}\>=\<\widehat f,\widehat{\phi}_{P}\>$. Now let
parametrize
the centers of the dyadic intervals $I$  of length $L$  by
$-(k-\frac 12)L$, $k\in \bbZ$.
Set $g_J(\omega): =[\sqrt{\varphi_J} \widehat f](c(J)+L^{-1}\omega)
e^{\pi i\omega}$ and note that $g_J$ is supported in $[-\frac{13}{50},
\frac{13}{50}]$. The left hand side of \eqref{expansion}
is equal to
\begin{align*}
&\sqrt{\varphi_J(\xi)}
\sum_k \int  \widehat f(\eta) \sqrt{\varphi_J(\eta)}
e^{ -2\pi i (k L-\tfrac L2) \eta}
L \,d\eta
\,
e^{ 2\pi i (k L-\tfrac L2) \xi}
\\&=\sqrt{\varphi_J(\xi)}
e^{-\pi i L(\xi-c(J))}
\sum_{k\in \bbZ} \int_{-1/2}^{1/2} g_J(\omega)
e^{-2\pi ik\omega} \,
d\omega  \,
e^{2\pi i k L (\xi-c(J))}
\\&= \sqrt{\varphi_J(\xi)}
e^{-\pi i L(\xi-c(J))}
g_J (L(\xi-c(J))) \,=\,
\widehat f(\xi)\varphi_J(\xi)
\end{align*}
which gives \eqref{expansion}. This in turn
yields the representation of $\cS'[f]$ in terms of wave packets:
\Be{wavepacketrep}
\CS'[f](x) = \sum_{k = 1}^{K} \Big(\sum_{J \in \J_{\xi_{k-1}(x),\xi_{k}(x)}} \sum_{|I| =
(1/(2|J|))}\<f,\phi_{I \times J}\>\phi_{I \times J}\Big)a_{k}(x).
\Ee

For the function $f$ under consideration the above Fourier series expansion
converges in  $L^2$-Sobolev spaces of arbitrary high order and thus
 the convergence in \eqref{wavepacketrep} is uniform for  $x\in
 [-A,A]$. Therefore it suffices,  for any {\it finite}
 family $\fP$ of tiles, to consider
the
operator $\cS''$ defined by
\Be{wavepacketrepfin}
\CS''[f](x) = \sum_{k = 1}^{K} \Big(\sum_{
\substack {(I,J) \in \fP \\
J \in \J_{\xi_{k-1}(x),\xi_{k}(x)}}}\<f,\phi_{I \times J}\>
\phi_{I \times J}(x)\Big)a_{k}(x).
\Ee

The wave packets will be sorted into a finite number of classes, each
well suited for further analysis. Sorting is accomplished by dividing
every $\J_{\xi,\xi'}$ into a finite number of disjoint sets. These
sets will be indexed by a fixed subset of $\{1,2,3\}
\times\{1,2,3,4\}^2\times\{\text{left},\ \text{right}\}.$
Specifically, for each $(m,n,side) \in
\{1,2,3,4\}^2\times\{\text{left},\ \text{right}\}$, we define

\begin{itemize}
\item{$\J_{\xi,\xi',(1,m,n,side)}=$} $\{J \in \D : J \subset (\xi,\xi')$, $\xi$ is in the interval $J - (m+1)|J|,$ $\xi'$ is in the interval $J + (n+1) |J|$, and $J$ is the $side$-child of its dyadic parent$\}$.
\item{$\J_{\xi,\xi',(2,m,n,side)}=$} $\{J \in \D : J \subset (\xi,\xi')$, $\xi$ is in the interval $J - (m+1) |J|,$ $\dis(\xi',J) \geq n|J|,$ and $J$ is the $side$-child of its dyadic parent$\}$.
\item{$\J_{\xi,\xi',(3,m,n,side)}=$} $\{J \in \D : J \subset (\xi,\xi')$, $\dis(\xi,J) > m|J|,$ $\xi'$ is in the interval $J + (n+1) |J|$, and $J$ is the $side$-child of its dyadic parent$\}$.
\end{itemize}

We will choose $R \subset \{1,2,3\} \times\{1,2,3,4\}^2\times\{\text{left},\ \text{right}\}$ so that for each $\xi,\xi'$, the collection $\{\J_{\xi,\xi',\rho}\}_{\rho \in R}$ is pairwise disjoint and $\J_{\xi,\xi'} = \cup_{\rho \in R} \J_{\xi,\xi',\rho}.$ We will also assume that for each $\rho \in R$ there is an $i(\rho) \in \{-1,0,1\}$ such that $|J'| = 2^{i(\rho)}|J|$ for every $\xi < \xi'$, $J \in \J_{\xi,\xi',\rho}$ and $J' \in \J_{\xi,\xi'}$ with $J'$ strictly to the left of $J$ and $\dis(J,J') = 0.$ One may check that these conditions are satisfied, say, for
\begin{multline*}
R = \{(1,2,1,\text{left}),(1,2,2,\text{left}),(1,3,1,\text{left}),\\(1,3,2,\text{left}),(2,1,1,\text{left}),(2,1,1,\text{right}),(2,2,1,\text{right}),\\(3,4,1,\text{left}),(3,3,1,\text{right}),(3,4,2,\text{left})\}.
\end{multline*}

It now follows that
\[
\CS''[f] = \sum_{\rho \in R} \CS^\rho[f]
\]
where
\[
\CS^\rho[f](x) = \sum_{k = 1}^{K} \bigg(\sum_{\substack{(I\times J) \in \fP
\\J \in \J_{\xi_{k-1}(x),\xi_{k}(x),\rho}}}\<f,\phi_{I \times J}\>\phi_{I \times J}\bigg)a_{k}(x).
\]

It will be convenient to rewrite each operator $\CS^\rho$ in terms of
{\it multitiles}. A  multitile will be a subset of $\rea^2$ of
the form $I \times \omega$ where $I \in \D$ and where $\omega$ is the
union of
three intervals $\omega_l,\omega_u,\omega_h$ in $\D$. For each
$\rho = (i,m,n,side) \in R$, we consider a set of $\rho$-multitiles which is parameterized by $\{(I,\omega_u) : I,\omega_u\in\D, |I||\omega_u| = 1/2, \text{\ and\ }\omega_u\text{\ is the\ }side\text{-child of its parent}\}.$ Specifically, given $\omega_u = [a,b)$
\begin{itemize}
\item{If $\rho = (1,m,n,side)$ then} $\omega_l = \omega_u - (m+1)|\omega_u|$ and $\omega_h = \omega_u + (n+1)|\omega_u|.$
\item{If $\rho = (2,m,n,side)$ then} $\omega_l = \omega_u - (m+1)|\omega_u|$ and $\omega_h = [a + (n+1)|\omega_u|,\infty).$
\item{If $\rho = (3,m,n,side)$ then} $\omega_l = (-\infty, b - (m+1)|\omega_u|)$ and $\omega_h = \omega_u + (n+1)|\omega_u|.$
\end{itemize}
For $i=1,2,3$ we shall say that $\rho$ is an $i$-index if
$\rho = (i,m,n,side)$.
For every $\rho$-multitile $P$, let $a_P(x) = a_k(x)$ if $k$ satisfies $1 \leq k \leq K$ and $\xi_{k-1}(x) \in \omega_l$ and $\xi_k(x) \in \omega_h$ (such a $k$ would clearly be unique), and $a_P(x) = 0$ if there is no such $k$. Then,
\[
\CS^\rho[f](x) =  \sum_{P \in \P_\rho} \<f,\phi_P\>\phi_{P}(x) a_{P}(x)
\]
where, for each $\rho$-multitile $P$, $\phi_P(x) = \sqrt{|I|}\cF^{-1}[
\sqrt{\varphi_{\omega_u,i(\rho)}}](x - c(I))$ and
$\P_\rho$ denotes the set of all  $\rho$-multitiles for which $I\times \omega_u$ belongs to $\fP$.

Inequality (\ref{linearizedtheorem}) and hence Theorem \ref{maintheorem} will then follow after proving the bound
\begin{equation} \label{csrhobound}
\|\CS^\rho[f]\|_{L^{p,\infty}} \lc \|f\|_{L^{p,1}}
\end{equation}
for each $\rho \in R$.
We shall only give the proof of this estimate for the case that $\rho$ is a
$1$-index  or $\rho$ is a $2$-index, and the case where $\rho$ is a $3$-index can be deduced by symmetry considerations.
Indeed, if $P=(I,\om_u):=([a,b), [c,d))$ and $\widetilde P:= ([-b,a), [-d,c))$
then
$\inn{f}{\phi_P}\phi_P=
\inn{f(-\cdot)}{\widetilde \phi_{\widetilde P}}\widetilde \phi_{\widetilde P}$
where, in the definition of the
$\widetilde \phi_{\widetilde P}$ the function
$\nu_i$ in \eqref{nu-functions}
is replaced with $\nu_i(-\cdot) $ (both are supported in $(-\frac {13}{25},
\frac {13}{25})$).
Now reflection sends a half open
interval $[a,b)$  to a half open interval $(-b,-a]$, however this
plays no role in our   symmetry argument if, as we do,
we assume that the set $\Xi$
does not contain endpoints of dyadic intervals.
We then see that the estimation of $S^\rho[f]$ for
$\rho=(3,m,n,side)$ is equivalent to the estimation of a
$S^{\tilde\rho}[f(-\cdot)]$  where the corresponding set $\{P\}$ of multitiles
is replaced with a set $\{\widetilde P\}$  of index
$\tilde \rho=(2,n,m,\text{\it opposite side})$ and
the set $\Xi$ is replaced with $\{\xi: -\xi\in \Xi\}$.
Beginning with \eqref{wavepacketadaptedbound}
both $\nu_i$ or $\nu_i(-\cdot)$ are allowed in
 the definition of  the functions $\varphi_{J,i}$ and $\phi_P$.


By the usual characterization of $L^{p,1}$ as superpositions of
functions bounded by characteristic functions it suffices to show
that
\[
\meas\big(\{x: |S^\rho [f](x)|>\lambda\}\big) \leq C^p \la^{-p}  |F|
\]
where  $F \subset \rea$, $|F|>0$,  $|f| \leq \indic{F}$, $\lambda > 0$, $2 < r < \infty$, and $r' \leq p < (1/2 - 1/r)^{-1}.$
This is accomplished by  proving that
for every measurable  $E \subset \rea,$
\begin{equation} \label{halfgbound}
\meas\Big(\big\{x\in E: |S^\rho[f](x)| > C \big(|F|/|E|\big)^{1/p}
 \big\}\Big)
 \leq |E|/2.
\end{equation}
Indeed, if we set $E_\la=\{x: |S^\rho [f](x)|>\lambda\}$ then by the
finiteness of the set of tiles under consideration the set
$E_\la$ has a priori finite measure. If $|E_\la| \leq C^p \la^{-p}  |F|$
then there is nothing to prove.
If the opposite inequality $|E_\la| > C^p \la^{-p}  |F|$ were true
then $\la> C(|F|/|E_\la|)^{1/p} $ and inequality
\eqref{halfgbound} applied to $E=E_\la$ would yield that $|E_\la|\le
|E_\la|/2$, a contradiction.

We finally note that, after possibly rescaling, we may assume
that $1 \leq |E| \leq 2$ in \eqref{halfgbound}.
The next four  sections will be devoted to the proof of  inequality
\eqref{halfgbound} in this case.

\section{Energy and density} \label{energydensitysection}
Recall that $S^{\rho}[f](x)=
\sum_{P} \<f,\phi_P\>\phi_{P} a_{P}$
where $P$ ranges over an arbitrary finite collection of
$\rho$-multitiles, $\rho$ is a $1$ or $2$-index. It is our goal to
show \eqref{halfgbound} and for this and the next chapter we fix the
function $f$ with $|f|\le \indic{F}$ and the set $E$.

Fix $1 \leq C_3 < C_2  < C_1$, with $C_2\in \bbN$,  such that for every multitile $P$,
\begin{equation*}
\begin{gathered}
\text{supp}(\widehat{\phi}_P) \subset C_3 \omega_u
\\
C_2 \omega_u \cap C_2 \omega_l = \emptyset,\qquad  C_2 \omega_u \cap
\omega_h = \emptyset,
\\
 C_2 \omega_l \subset C_1 \omega_u, \qquad C_2 \omega_u \subset C_1
 \omega_l;
\end{gathered}
\end{equation*}
recall that dilations of finite
 intervals are with respect to their center.
One may check that the values $C_3 = 11/10, C_2 = 2,$ and $C_1 = 12$
 satisfy all  these properties.

The wave packet  is adapted 
to the multitile $P$. As $\widehat \phi_P$ is compactly supported (in
$ C_3 \omega_u$) the function $\phi_P$ cannot have compact support,
but as a replacement we have the following bounds
involving
\[
w_{I}(x): =  \frac{1}{|I|}\left(1 + \frac{|x - c(I)|}{|I|}\right)^{-N}
\]
for a fixed large $N\gg 10$; namely
\begin{equation} \label{wavepacketadaptedbound}
\Big|\frac{d^n}{dx^n} \big(\exp(-2\pi i c(\omega_u) \cdot \big)\phi_P\big)(x)\Big| \leq C'(n) |I|^{(1/2) -n} |w_I(x)|
\end{equation}
for each  $n \geq 0.$

We are working with a given finite set of $\rho$-multitiles  $\P$
(with $\rho$ a $1$- or $2$-index) and we let $M_\circ=M_\circ(\P)$ be the smallest integer $M$ for which all tiles are contained in the square
$[-2^M, 2^M]^2$. Throughout this paper  we fix
$$\cXtop=\{\eta: |\eta|\le C_1 2^{M_\circ+10},\,\,
\eta=
n 2^{-M_\circ-10}, \text{ for some $n\in \bbZ$ } \}$$
as the set of admissible  top-frequencies for trees, as 
in the following definition.

\begin{definition}
Consider a triple    $\cT=(T_\cT, I_\cT, \xi_\cT)$, 
 with   a set of multitiles $T_\cT$, a dyadic interval  $I_\cT
\subset [-2^{M_\circ},2^{M_\circ})$  and a point $\xi_\cT \in \cXtop$.
We say that $\cT$ is a tree if the following properties are satisfied.

(i) $I\subset I_\cT$ for all $P=(I,\om_u)\in T_\cT$.

(ii) If $P=(I,\om_u)$ and $\om_m$ denotes the convex hull
  of $C_2 \omega_u \cup C_2 \omega_l$  then
$$\om_\cT:=
\big[\xi_\cT - \tfrac{C_2 - 1}{4|I_\cT|}, \xi_\cT + \tfrac{C_2 -    1}{4|I_\cT|} \big)$$
is contained in  $\om_m$.

We refer to $I_\cT$ as the top interval of the tree, and to $\xi_\cT$ as the top frequency of the tree.
\end{definition}

In order not to overload the notation we usually  refer to the set $T$
as ``the'' tree (keeping in mind that it carries additional information
of a top frequency and a top interval),
and we shall also use the notation $I_T$, $\xi_T$ and $\om_T$ in place  of
$I_\cT$, $\xi_\cT$ and $\om_\cT$.
With this convention we also define
\begin{definition} ${}$

(i) A
  tree $(T,I_T,\xi_T)$ is
{\it $l$-overlapping} if $\xi_T \in C_2 \omega_l$ for every $P
  \in T$.

(ii) A tree $(T,I_T,\xi_T)$ is
{\it $l$-lacunary} if $\xi_T \not\in C_2 \omega_l$ for every $P \in T.$

\end{definition}

Notice that the union of two trees with the  top data $I_T, \xi_T$
is again a tree with the same top data.
Also,
the union of two $l$-overlapping  trees with the same top data
is again an $l$-overlapping  tree with the same top data.

We split our
finite collection of multitiles into a bounded number of
subcollections
satisfying certain separation conditions (i.e. henceforth all
multitiles will be assumed to belong to a fixed subcollection).

\medskip

\noi{\bf Separation assumptions.}
\begin{equation} \label{separationofscales}
\text{If\ } P,P' \text{\ satisfy\ } |\omega_u'| < |\omega_u|, \text{\ then\ }
|\omega_u'| \leq \frac{C_2 - C_3}{2C_1} |\omega_u|.
\end{equation}
\begin{equation} \label{separationprop2}
\text{If\ }P,P' \text{\ satisfy\ } C_1\omega_u \cap C_1 \omega_u' \neq \emptyset \text{\ and\ } |\omega_u| = |\omega_u'| \text{\ then\ } \omega_u = \omega_u'.
\end{equation}

\medskip

As immediate but important  consequence of the  separation assumptions
is
the frequently used
\begin{observation}\label{treeobs} Let $T$ be a tree satisfying the separation
  properties \eqref{separationofscales} and \eqref{separationprop2}.
Then the following properties hold.

(i)
If $T$ is an $l$-overlapping tree, $P,P' \in T$, and $|\omega_u'|
< |\omega_u|$ then $C_3 \omega_u \cap C_2 \omega_u' = \emptyset.$

(ii) 
If $T$ is an $l$-lacunary tree, $P,P' \in T$,
and $|\omega_u'| <
|\omega_u|$ then $C_3 \omega_l \cap C_3 \omega_l' = \emptyset.$

(iii) 
If $P,P' \in T$, $P\neq P'$,
and $|\omega_u| = |\omega_u'|$, then $I \cap I' = \emptyset.$
\end{observation}

\medskip

As in previous proofs of Carleson's theorem (in particular \cite{lacey00pbc})
we shall split the set of multitiles into subsets with controllable
{\it energy} and {\it density} associated to the function $f$ and the
set $E$, respectively. Here we work with the following definitions.
\begin{definition} Fix $f\in L^2(\bbR)$ and a measurable set $E\subset \bbR$.
Given any collection of multi\-tiles $\P$  we define
\[\energy(\P) = \sup_{T}
\sqrt{
\frac{1}{|I_T|} \sum_{P \in T}
|\inn {f}{ \phi_P}|^2
}\]
where the $\sup$ ranges over all $l$-overlapping trees $T \subset \P$.

Given a measurable set $E\subset \bbR$ we set
\[\density(\P )
=  \sup_T \Big( \frac{1}{|I_T|}\int_E  \Big(1 + \frac{|x - c(I_T)|}{|I_T|}\Big)^{-4}\sum_{k =1}^{K} |a_k(x)|^{r'} \indic{\omega_T}(\xi_{k-1}(x)) \ dx
 \Big)^{1/r'}
\]
where the $\sup$ is over all non-empty trees $T \subset \P$.
\end{definition}

\begin{remarknotag}
Concerning the terminology, one can argue
that the squareroot should be omitted in the definition of an energy.
However we work with the above definition to conform to
  \cite{lacey00pbc} and other
  papers in time-frequency analysis.
\end{remarknotag}

\begin{lemma}\label{univ-den-ene-bound}
Let $|f|$ be bounded by $\indic{F}$.
For any family $\P$ of multitiles 
the density of  $\P$ (with respect to
the set $E$) and the
energy (with respect to $f$) are  bounded by a universal constant.
\end{lemma}

\begin{proof}
Clearly the density is bounded by $\int_{\bbR}(1+|x|)^{-4} \, dx<3$.
Concerning the energy bound
we let  $T$ be any $l$-overlapping tree, and split $f = f' + f''$
where $f' = \indic{3I_T} f$. We estimate
$\sum_{P \in T} |\<f',\phi_P\>|^2
\le \|f'\|_{L^2}(\sum_{P \in T} |\<f',\phi_P\>|^2)^{1/2}$, by the Cauchy-Schwarz inequality. Now use
that the supports of the $\widehat \phi_P$ are disjoint for
different sizes of frequency intervals, and then,  for a fixed size, use
 the bounds
 \eqref{wavepacketadaptedbound} (for $n=0$)  to see that
\[\Big\|\sum_{P \in T} \<f',\phi_P\>\phi_P  \Big\|_{L^2}\le
\Big(\sum_{j\in \bbZ} \Big\|\sum_{\substack{P \in T \\ |\om_u|=2^{-j}}} \<f',\phi_P\>\phi_P  \Big\|_{L^2}^2\Big)^{1/2}
\lc
\Big(\sum_{P \in T} |\<f',\phi_P\>|^2  \Big)^{1/2}.\]
Hence,
\[\sum_{P \in T} |\<f',\phi_P\>|^2   \lc \|f'\|_{L^2}^2
\le |F\cap 3I_T|\lc |I_T|.\]
Furthermore, since $|f''| \leq \indic{\rea \setminus 3I_T}$, we have the estimate
\[
|\<f'',\phi_P\>| \lc |I|^{1/2}(1 + \dis(I,\rea \setminus 3I_T)/|I|)^{-(N-1)} \lc |I|^{1/2} (|I|/|I_T|)^{N-1}
\]
Summing in $P$, we obtain
\[
\sum_{P \in T} |\<f'',\phi_P\>|^2 \lc |I_T|.
\]
Combining the estimates for $f'$ and $f''$ we see that
$\sum_{P \in T} |\<f,\phi_P\>|^2 \lc |I_T|$ for every $l$-overlapping
tree and it follows that
the energy of $\P$ with respect to $f$ is bounded above by a universal 
constant.
\end{proof}

The following proposition allows one to decompose an arbitrary
collection of multitiles into the union of trees, where the trees are
divided into collections $\T_j$ with the energy of trees from $\T_j$
bounded by $2^{-j}.$ The control over energy is balanced by an $L^q$
bound for the functions $N_{j,\ell} := \sum_{T \in \T_j} \indic{2^\ell
  I_T}.$ In contrast to \cite{lacey00pbc} and \cite{grafakos04bmd}, it
is necessary here to consider $q > 1$ and $\ell > 0$ in order to
effectively use the
tree estimate Proposition \ref{treeestimate} with $q > 1$. Note that such
$L^q$ bound for $N_{j,\ell}$ is established by combining (\ref{etreeintervalbound})
and (\ref{etreebmobound}) below.
The bound
(\ref{pisatree}) permits one to make further decompositions to take
advantage of large $|F|$ in the $L^q$ bound for the $N_{j,\ell}$ while
maintaining compatibility with bounds for trees with a fixed density
obtained from Proposition \ref{densityincrementprop}.

\begin{proposition} \label{energyincrementprop}
Let $\ene>0$,
let $|f|$ be bounded above by $\indic{F}$  and let $\P$ be a collection of multitiles with energy bounded above by
$\ene.$
Then, there is a collection
of
trees $\T$ such that
\begin{equation} \label{etreeintervalbound}
\sum_{T \in \T} |I_T| \lc \ene^{-2} |F|
\end{equation}
and
\[
\energy\left( \P \setminus \bigcup_{T \in \T} T\right) \leq \ene/2,
\]
and such that, for every integer $\ell \geq 0$,
\begin{equation} \label{etreebmobound}
\Big\|\sum_{T \in \T} \indic{2^\ell I_T}\Big\|_{BMO} \lc 2^{2\ell} \ene^{-2}.
\end{equation}
Furthermore, if for some collection of trees $\T'$,
\begin{equation} \label{pisauniont}
\P = \bigcup_{T' \in \T'}T'
\end{equation}
then
\begin{equation} \label{pisatree}
\sum_{T\in \T} |I_T|\lc\sum_{T' \in \T'} |I_{T'}|.
\end{equation}
\end{proposition}
Above, and subsequently, $\|\cdot\|_{BMO}$ denotes the dyadic $BMO$ norm.

\begin{proof}
We select trees through an iterative procedure.
First, if $\energy (\P)\le \ene/2$
then no tree is chosen and $\T=\emptyset$.


If $\energy (\P)\ge \ene/2$ then we
observe that
there is an $l$-overlapping  tree $S\subset \P$  for which
\begin{equation} \label{shaslargeenergy}
\frac{1}{|I_S|} \sum_{P \in S} |\< f , \phi_P\>|^2 \geq \ene^2/4.
\end{equation}
There is  only a finite number of such  trees and as $S_1$ we choose
one for
which the top datum $\xi_S$ is maximal (in $\bbR$).
Note that the maximality can be achieved as we restrict all top frequency  data to the finite set $\cXtop$.
Let $T_{1}$ be the  tree in $\P$ which has
top data $(\xi_{S_{1}}, I_{S_{1}})$  and which is maximal with
respect to inclusion.

Suppose that trees $S_k, T_k$ have been chosen for $k=1, \ldots, j$.
Set
\[
\P_j = \P \setminus \bigcup_{k=1}^{j}  T_k
\]
If $\energy(\P_j) \leq \ene/2$ then we terminate the procedure, set $\T = \{T_k\}_{1 \leq k \leq j}$ and $n=j.$ Otherwise, we may find an $l$-overlapping tree $S \subset \P_j$  such that \eqref{shaslargeenergy} holds.
Among $l$-overlapping in $\P_j$  satisfying
\eqref{shaslargeenergy} choose one with maximal top-frequency (in $\cXtop$)
and
label this tree
$S_{j+1}$.
Let $T_{j+1}$ be the maximal
tree in $\P_j$ which has
 top data $(\xi_{S_{j+1}}, I_{S_{j+1}})$  and which is maximal with
respect to inclusion.
 This process will eventually stop since each $T_j$ is nonempty and $\P$ is finite.

\medskip

\noi{\it Proof of (\ref{etreeintervalbound})}. It suffices to show
\Be{squarebyfct}
\Big(\frac{\ene^2}{|F|} \sum_{j=1}^{n} |I_{S_j}|\Big)^2 \lc \frac{\ene^2}{|F|} \sum_{j=1}^{n} |I_{S_j}|.
\Ee
Since the $S_j$ satisfy  (\ref{shaslargeenergy}), we have
\[
\Big(\frac{\ene^2}{|F|} \sum_{j=1}^{n} |I_{S_j}|\Big)^2 \leq 16 |F|^{-2}
\Big(\sum_{j=1}^{n} \sum_{P \in S_j}
|\langle f, \phi_P\rangle|^2 \Big)^{2}.
\]
Now
\begin{multline*}
\Big( \sum_{j=1}^n \sum_{P \in S_j}  |\inn{f}{\phi_P}|^2\Big)^2=
\Big(\biginn{\sum_{j=1}^n \sum_{P\in S_j}
\inn{f}{\phi_P}\phi_P} {f}\Big)^2
\\
\leq \|f\|_2^2\,\Big\|\sum_{j=1}^n\sum_{P \in S_j}
\inn{f}{\phi_P}\phi_P\Big\|_2^2\leq |F|
\sum_{j=1}^n \sum_{P\in S_j}
\sum_{k=1}^n \sum_{P'\in S_j}
|\inn{f}{\phi_P}|
\,|\inn{f}{\phi_{P'}}|\, |\inn{\phi_P}{\phi_{P'}}|
\end{multline*}
where in the last inequality we used $|f|\le \indic{F}$.
By symmetry, it remains, for
\eqref{squarebyfct}
to show that
\begin{equation} \label{etreeorth1}
\sum_{j=1}^{n} \sum_{k=1}^{n} \sum_{P \in S_j} \sum_{P' \in S_k : |I'| = |I|}  |\langle f, \phi_P\rangle| \,|\langle f, \phi_{P'} \rangle|\, |\<\phi_P,\phi_{P'}\>| \lc \ene^2 \sum_{j=1}^{n} |I_{S_j}|
\end{equation}
and
\begin{equation} \label{etreeorth2}
\sum_{j=1}^{n} \sum_{k=1}^{n} \sum_{P \in S_j} \sum_{P' \in S_k : |I'| < |I|} |\langle f, \phi_P\rangle| \,|\langle f, \phi_{P'} \rangle| \,|\<\phi_P,\phi_{P'}\>| \lc \ene^2 \sum_{j=1}^{n} |I_{S_j}|.
\end{equation}
In both cases, we will use the estimate
\begin{equation} \label{schwartzestimate}
|\<\phi_P,\phi_{P'}\>| \lc \Big( \frac{|I|}{|I'|} \Big)^{1/2} \< w_{I}
,
\indic{I'} \>.
\end{equation}
which holds whenever $|I'| \leq |I|.$

Estimating the product of two terms by the square of their maximum, we see that the left side of (\ref{etreeorth1}) is
\[
\leq 2 \sum_{j=1}^{n} \sum_{k=1}^{n} \sum_{P \in S_j} \sum_{P' \in S_k : |I'| = |I|}  |\langle f, \phi_P\rangle|^2 |\<\phi_P,\phi_{P'}\>|.
\]
Recall that $\<\phi_P,\phi_P'\> = 0$ unless $C_3 \omega_u \cap C_3 \omega_u' \neq \emptyset.$ Thus, by (\ref{separationprop2}), (\ref{schwartzestimate}) and the fact that the $S_k$ are pairwise disjoint, we can estimate
the last display  by
\[
2 \sum_{j=1}^{n} \sum_{P \in S_j} |\langle f, \phi_P\rangle|^2 \sum_{I' : |I'| = |I|} \< w_{I} , \indic{I'} \> \lc \sum_{j=1}^{n} \sum_{P \in S_j} |\langle f, \phi_P\rangle|^2\lc \sum_{j=1}^{n}  \ene^2 |I_{S_j}|.
\]
This finishes the proof of (\ref{etreeorth1}).

Applying the Cauchy-Schwarz inequality, we see that the left side of
(\ref{etreeorth2}) is bounded by
\[
\sum_{j=1}^{n}  \Big(\sum_{P \in S_j} |\langle f, \phi_P\rangle|^2
\Big)^{1/2} \Big(\sum_{P \in S_j} \Big( \sum_{k=1}^{n}\sum_{P' \in S_k : |I'| < |I|}  |\langle f, \phi_{P'} \rangle| |\<\phi_P,\phi_{P'}\>| \Big)^2 \Big)^{1/2}.
\]
Twice using the fact that the energy of $\P$ is bounded by $\ene$, we see that the last display is
\[
\leq \ene^2 \sum_{j=1}^{n} |I_{S_j}|^{1/2} \biggl(\sum_{P \in S_j} \Big( \sum_{k=1}^{n}\sum_{P' \in S_k : |I'| < |I|}   \big|\biginn{\phi_P}{ |I_{P'}|^{1/2}\phi_{P'}}\big|
\Big)^2 \biggr)^{1/2}.
\]
Thus, to prove (\ref{etreeorth2}) it remains to show that, for each $j$,
\begin{equation}\label{etreeorth22}
\sum_{P \in S_j} \Big( \sum_{k=1}^{n}\sum_{P' \in S_k : |I'| < |I|}
\big|\biginn{\phi_P}{ |I_{P'}|^{1/2}\phi_{P'}}\big| \Big)^2  \lc |I_{S_j}|.
\end{equation}
Again, we only have $|\inn{\phi_P}{ |I_{P'}|^{1/2}\phi_{P'}}|$ nonzero
when $C_{3}\omega_u \cap C_3 \omega_u' \neq \emptyset$ which can only happen if $\sup C_{3} \omega_u \in C_3 \omega_u'$ or $\inf C_{3} \omega_u \in C_3 \omega_u'.$ Applying (\ref{schwartzestimate}), we thus see that the left side of
\eqref{etreeorth22}
 is dominated by a constant times  the expression
\[\sum_{P \in S_j} |I_P|\Big( \sum_{k=1}^{n}
\sum_{\substack{P' \in S_k : |I'| < |I| \\ \sup C_3 \omega_u \in C_{3} \omega_u'}}   \<w_I, \indic{I'}\> \Big)^2 \\ +  \sum_{P \in S_j} |I_P|\Big( \sum_{k=1}^{n}\sum_{\substack{P' \in S_k : |I'| < |I| \\ \inf C_3 \omega_u \in C_{3} \omega_u'}}  \<w_I, \indic{I'}\> \Big)^2.
\]
We now claim that, for each $P\in S_j$ (with time interval $I$),
\Be{claimofdisjointness}
\Big( \sum_{k=1}^{n}
\sum_{\substack{P' \in S_k : |I'| < |I| \\ \sup C_3 \omega_u \in C_{3} \omega_u'}}   \<w_I, \indic{I'}\> \big)^2
\le \inn{w_I}{\indic {\bbR \setminus I_{S_j}}}
\Ee
and that  the same inequality with $\sup$ replaced by $\inf$ in the $P'$
 summation holds as well.
To see this consider two multitiles $P^1=(I^1,\om_u^1) \in S_{\ka_1}$
and $P^2=(I^2,\om_u^2)\in S_{\ka_2} $, $P^{1}\neq P^{2}$.
 so that
$|I^2|\le |I^1|$ and $C_3\om_u^1\cap C_3\om_u^2\neq \emptyset$.
The last condition implies $\ka_1\neq \ka_2$
(since $S_{\ka_1}$ and $S_{\ka_2}$ are $l$-overlapping).
The inequality \eqref{claimofdisjointness} is immediate if we can show
that $I^1$ and $I^2$ are disjoint and if in addition
$|I^2| < |I^1|,$ then $I_2$ does not belong to the top interval
of the tree $S_{\ka_1}$.
Now, if $|I^1| = |I^2|$, then
from (\ref{separationprop2}) it follows that $\omega_u^1 = \omega_u^2$
and hence, since $P^1 \neq P^2,$ we have $I^1 \cap I^2 = \emptyset.$
If $|I^2| < |I^1|,$
then by
(\ref{separationofscales})
$|\om_u^1|\le \frac{C_2-C_3}{2C_1}|\omega_u^2|$;
since
$C_3\om_u^1\cap C_3\om_u^2\neq \emptyset$ this implies that
$\inf C_2\omega_l^1> \sup C_2\omega_l^2$. As both trees are $l$-overlapping
the top  frequency of $S_{\ka_1}$  belongs  to $C_2\om_l^1$ and  is above
the  top frequency of $S_{\ka_2}$  which  belongs  to $C_2\om_l^2$.
Thus by the maximality condition on the top frequency in the selection process of the trees we see that the tree $S_{\ka_1}$ was selected before
the tree $S_{\ka_2}$, i.e. $\ka_1<\ka_2$. This implies that
$P_2$ does not belong to the tree $T_{\ka_1}$, and since the interval
$\omega_{S_{\ka_1}}$ is contained in  the convex hull of $\omega_l^2$ and
$\omega_u^2$ we see that the time intervals $I^2$ and
$I_{T_{\ka_1}}=I_{S_{\ka_1}}$
cannot intersect.
Thus $I^1\cap I^2=\emptyset$ as  $I^1\subset I_{T_{\ka_1}}$.
This concludes the argument for \eqref{claimofdisjointness}.

Now by the disjointness condition we see that indeed the left hand side of
\eqref{etreeorth22} is bounded by a constant times
\begin{align*}
\sum_{P \in S_j} |I_P| \biginn{w_I}{\indic{\rea \setminus I_{S_j}}}^2
&\lc \sum_{\ell : 2^\ell \leq |I_{S_j}|} 2^\ell \sum_{P \in S_j : |I| = 2^\ell}
\biginn{w_I}{\indic{\rea \setminus I_{S_j}}}
\\&
\lc |I_{S_j}| \sup_{\ell : 2^\ell \leq |I_{S_j}|}  \sum_{P \in S_j : |I| = 2^\ell}
\biginn{w_I}{\indic{\rea \setminus I_{S_j}}}
\end{align*}
One may check that, for each $\ell$,
\[
\sum_{P \in S_j : |I| = 2^\ell}
\biginn{w_I}{\indic{\rea \setminus I_{S_j}}} \lc 1
\]
and (\ref{etreeorth22}) follows. We have already seen that  (\ref{etreeorth22})
implies (\ref{etreeorth2}); this completes the proof of
(\ref{etreeintervalbound}).

\medskip

\noi{\it Proof of (\ref{etreebmobound})}.  We need to show that for each dyadic interval
$J$, we have
\[
\frac{1}{|J|} \int_{J} \Big|\sum_{T \in \T}\indic{2^\ell I_T}(x) -
\frac{1}{|J|}\int_{J} \sum_{T \in \T} \indic{2^\ell I_T}(y)\ dy\Big|\
dx \lc 2^{2\ell} \ene^{-2}.
\]
This is an immediate consequence of
\begin{equation} \label{bmosummeasures}
\sum_{\substack{T \in \T \\ J \cap 2^\ell I_T \neq \emptyset,J}} |I_T|
\lc
\ene^{-2} 2^\ell |J|.
\end{equation}
Let \[\widetilde{\T} = \{T \in \T : I_T \subset 2^{\ell+1}J, |I_T| \leq |J|\}\] and note that
if $T \in \T$ with $2^\ell I_T \cap J \neq \emptyset,J$ then
$T \in \widetilde{\T}.$
Write $f = f' + f''$ where $|f'| \leq \indic{F\cap 2^{\ell+5}J}$ and
$|f''| \leq \indic{F\cap\rea \setminus 2^{\ell+5}J}.$

We will write $\widetilde{\T}$ as the union of collections of trees
$\T^{\text{main}} \cup \T^0 \cup \T^1 \cup \ldots$
each of which will have certain properties related to the energy.
For each tree $T \in \widetilde{\T}$ there is an l-overlapping tree
$S=S(T)$
chosen in the algorithm above with $I_S = I_T$ and
\begin{equation} \label{shaslargeenergy2}
\frac{1}{|I_S|}\sum_{P \in S} |\langle f, \phi_{P} \rangle|^2 \geq  \ene^2/4.
\end{equation}
Let \[\T^0 = \{T \in \widetilde{\T} : \frac{1}{|I_S|}\sum_{P \in S} |\langle f'', \phi_{P} \rangle|^2 \geq  \ene^2/16\}.\]
For $j \geq 1$, define
\[
\T^{j} = \{T \in \widetilde{\T}  : \sup_{\substack{S' \subset S(T) \\ |I_{S'}| = 2^{-j}|I_T|}} \frac{1}{|I_{S'}|}\sum_{P \in S'} |\langle f'', \phi_{P} \rangle|^2 \geq \ene^2/16  \}
\]
where, for each $T$, the $\sup$ above is taken over all
$l$-overlapping trees $S'$ with $S' \subset S(T)$. Finally, let
\[\T^{\text{main}} = \{T \in \tilde{\T} \setminus (\T^0 \cup \T^1 \cup \ldots)\}.\]
We split the sum \eqref{bmosummeasures} into the ``main'' term involving trees in
$\T^{\text{main}}$ and an error term involving $\cup_{j\ge 0}\T_j$.

We first consider the main term. It is our objective to prove
\Be{sumITinTprime}
\sum_{T \in \T^{\text{main}}} |I_T| \lc \ene^{-2} 2^\ell |J|.
\Ee
Let $T \in \T^{\text{main}}$ and let $S'$ be any $l$-overlapping tree
contained in $S$ satisfying $|I_{S'}| \leq |I_S|$. Since the energy of $\P$ is bounded by $\ene$ and since $T$ is not in any $\T^j$, we have
\[
\frac{1}{|I_{S'}|}\sum_{P \in S'} |\langle f', \phi_P \rangle|^2 \leq 2 \frac{1}{|I_{S'}|}\sum_{P \in S'} |\langle f, \phi_P \rangle|^2 + 2 \frac{1}{|I_{S'}|}\sum_{P \in S'} |\langle f'', \phi_P \rangle|^2 \lc \ene^2.
\]
From (\ref{shaslargeenergy2}) and the fact that $T \notin \T^0$, we have
\[
\frac{1}{|S(T)|}\sum_{P \in S(T)} |\langle f', \phi_P \rangle|^2 \geq \ene^2/8 - \ene^2/16 = \ene^2/16.
\]
This inequality allows us to essentially repeat the above proof of
(\ref{etreeintervalbound}). The $l$-overlapping trees $S(T)$ form a
(finite or infinite)
subsequence
of the sequence $S_j$ which we denote by $S_{j(\nu)}$ so that we have
$j(\nu)>j(\nu')$ for $\nu>\nu'$.
We need to prove the analogue of \eqref{squarebyfct}
which is
\Be{squarebyfctanalog}
\frac{\ene^2}{|F\cap  2^{\ell+5}J|}  \sum_{\nu=1}^{n}
|I_{S_{j(\nu)}}|\lc 1
\Ee
and as before we are aiming to estimate the square of the expression
on the left hand side by the expression itself.
The  $S_{j(\nu)}$ satisfy
\[
\frac{1}{|S_{j(\nu)}|}\sum_{P \in S_{j(\nu)}} |\langle f', \phi_P \rangle|^2 \geq  \frac{\ene^2}{16}
\]
and therefore
\[
\Big(\frac{\ene^2}{|F\cap 2^{\ell+5}J|} \sum_{\nu=1}^{n}
|I_{S_{j(\nu)}}|\Big)^2 \leq 256 |F\cap 2^{\ell+5}J|^{-2}
\Big(\sum_{\nu=1}^{n} \sum_{P \in S_{j(\nu)}}
|\langle f, \phi_P\rangle|^2 \Big)^{2}.
\]
We continue to argue
with exactly the same reasoning as in the proof of
(\ref{squarebyfct}),
replacing $f$ with $f'$ and $F$ with
$F\cap 2^{\ell+5}J$. This leads to the proof of
 \eqref{squarebyfctanalog} and thus to
\[
\sum_{T \in \T^{\text{main}}} |I_T| \lc \ene^{-2} |F\cap 2^{\ell+5} J|
\]
which is clearly  $\lc \ene^{-2} 2^\ell |J|$. Thus
\eqref{sumITinTprime} is established.

For the complimentary terms we prove better estimates, namely, for $j=0,1,2,\dots$,
\Be{sumTj}
\sum_{T \in \T_j} |I_T| \lc 2^{-\ell-j} \ene^{-2} |J|.
\Ee

For each
$T\in \T^j$ and $P \in S(T)$ we have $|I|\le 2^{-j} |I_T|\le 2^{-j}|J|.$ Thus 
\[
\sum_{T \in \T^j} |I_T| \lc
 \sum_{T \in \T^j} 2^j \ene^{-2} \sum_{P \in S (T): |I| \leq 2^{-j}|J|} |\langle f'', \phi_P\rangle|^{2}.
\]
Since the $S$  are pairwise disjoint, the right hand side  is
\[\lc 2^j \ene^{-2} \sum_{k \geq j} \sum_{\substack{P : |I| = 2^{-k}|J| \\ I \subset
2^{\ell+1}J}} |\langle f'', \phi_P\rangle|^{2}.
\]
Fixing $k\ge j$, we apply Minkowski's inequality to obtain
\[
\sum_{\substack{P : |I| = 2^{-k}|J| \\ I \subset 2^{\ell+1}J}} |\langle f'', \phi_P\rangle|^{2} \leq \biggl(\sum_{\substack{K : |K| = 2^{1-k}|J|\\K \cap 2^{\ell+2}J = \emptyset}} \Big( \sum_{\substack{P : |I| = 2^{-k}|J| \\ I \subset 2^{\ell+1}J}}
|\langle \indic{K} f'', \phi_P\rangle|^{2}\Big)^{1/2}\biggr)^2
\]
where above, we sum over dyadic intervals $K$ and use the fact that
$f''$ is supported on $\rea \setminus 2^{\ell+5}J.$

Now note that
$\phi_{P'} = c \exp(2\pi i (c(\omega_u') - c(\omega_u) )\cdot)
\,\phi_{P}$ when $I = I'$. Thus  if $\cP(I_\circ)$ is a collection of
disjoint
multitiles with common
time interval $I_\circ$, if $g$ is supported on $CI_\circ$  and if $P_{I_\circ}$ is any fixed multitile in
$\cP(I_\circ)$ then Bessel's inequality gives
\[
\sum_{P\in \cP(I_\circ)} |\inn{g}{\phi_P}|^2 \lc |I_\circ| \int
|g\phi_{P_{I_\circ}}|^2 dx.
\]
We apply this observation to the inner sums in the previous display
and obtain the inequality
\Be{sumafterorth}
\sum_{\substack{P : |I| = 2^{-k}|J| \\ I \subset 2^{\ell+1}J}}
 |\langle f'', \phi_P\rangle|^{2}
\lc 2^{-k}|J| \biggl(\sum_{\substack{K : |K| = 2^{1-k}|J|\\K \cap 2^{\ell+2}J
= \emptyset}} \Big( \sum_{\substack{I : |I| = 2^{-k}|J| \\ I \subset 2^{\ell+1}J}}
\big\| \indic{K} f'' \phi_{P_I}\big\|_{L^2}^{2}\Big)^{1/2}\biggr)^2
\Ee
where for each $I$, $P_I$ is any multitile with time interval
$I$. Since $|f''|\le 1$ the bound  (\ref{wavepacketadaptedbound}) yields
\[\big\|  \indic{K} f'' \phi_{P_I}\big\|_{L^2}^{2}\lc\big(1 +
\tfrac{\dis(K,I)}{|I|}\big)^{-N}.\] Applying it with large $N$  we
see that
the right hand side of \eqref{sumafterorth}  is
$\lc 2^{-(k+\ell)(N-4)} 2^{-k}|J|$. Summing over $k\ge j$
 we obtain inequality  \eqref{sumTj}. This concludes the proof of
(\ref{etreebmobound}).

\medskip

\noi{\it Proof of $(\ref{pisatree})$}.
For each $T \in \T$, let $S\equiv S(T)$ be the corresponding $l$-overlapping tree from the selection algorithm above and recall
\[
\sum_{T \in \T} |I_{T}| (\ene/2)^2 \leq  \sum_{P \in \bigcup_{T \in \T} S}|\<f, \phi_P\>|^2.
\]
Since $\P = \bigcup_{T' \in \T'} T'$, the right side above is
dominated by
\begin{multline}\label{threeterms}
 \sum_{T' \in \T'} \sum_{\substack{P \in T' \cap \bigcup_{T \in \T} S \\ \xi_{T'} \in C_2 \omega_l}} |\<f, \phi_P\>|^2 + \sum_{T' \in \T'} \sum_{\substack{P \in T' \cap \bigcup_{T \in \T} S \\ \xi_{T'} \geq \inf C_3 \omega_u}} |\<f, \phi_P\>|^2 \\ + \sum_{T' \in \T'} \sum_{\substack{P \in T' \cap \bigcup_{T \in \T} S \\ \sup C_2 \omega_l \leq \xi_{T'} < \inf C_3 \omega_u}} |\<f, \phi_P\>|^2.
\end{multline}
For each $T'\in \T'$ the set of tiles $P\in T'\cap \bigcup_{T\in\T}
S(T)$ with the property that $\xi_{T'}\in C_2\omega_l$ is by
definition an $l$-overlapping tree. Thus, since  $\P$ has energy
bounded by
$\ene$, we estimate the first term in \eqref{threeterms}
\[
\sum_{T' \in \T'} \sum_{\substack{P \in T' \cap \bigcup_{T \in \T} S \\ \xi_{T'} \in C_2 \omega_l}} |\<f, w_u\>|^2 \leq \sum_{T' \in \T'} \ene^2 |I_{T'}|.
\]

For the estimation of the second   term in \eqref{threeterms} we observe
that any fixed multitile forms an $l$-overlapping tree (with respect
to some top data), and we use the energy bound for fixed tiles.
We observe that the rectangles
$\{I \times [\inf C_3 \omega_u,\sup C_2 \omega_u)
 : P \in \bigcup_{T \in \T} S\}$ are pairwise disjoint.
Indeed if $P\in S $, $P'\in S$ this follows since $S$ is $l$-overlapping
(\cf. Observation \ref{treeobs}). If $P\in S_{\ka_1}$ and $P'\in S_{\ka_2}$
$\ka_1<\ka_2$ then $\xi_{S_{\ka_1}}\ge \xi_{S_{\ka_2}}$ and  an overlap of $[\inf C_3 \omega_u,\sup C_2 \omega_u)$
and $[\inf C_3 \omega_u',\sup C_2 \omega_u')$ would imply that $P'$
belongs  to the maximal tree with the same top data as $S_{\ka_1}$, i.e. this would imply that  $P'\in T_{\ka_1}$ which is disjoint from $S_{\ka_2}$.

This allows us to estimate for any fixed $T'\in \T'$
\[
\sum_{\substack{P \in T' \cap \bigcup_{T \in \T} S
    \\ \xi_{T'} \geq \inf C_3 \omega_u}} |\<f, \phi_P\>|^2 \leq  \sum_{\substack{P \in T' \cap \bigcup_{T \in \T} S \\
    \xi_{T'} \geq \inf C_3 \omega_u}}
\ene^2 |I| \leq  \ene^2 |I_{T'}|.
\]
Now sum over $T'\in \T'$ and it follows that the middle term in
\eqref{threeterms} is $\lc \sum_{T'\in\T'}\ene^2 |I_{T'}|.$
     and summing


For the estimation of the third term in \eqref{threeterms} we begin with a preliminary
observation, also related to the selection of the $S$.
Suppose $P \in T' \cap S$, $\tilde{P} \in T' \cap \tilde{S}$ where $T,\tilde{T} \in \T$ and
\[\xi_{T'} \in [\sup C_2 \omega_l,\inf C_3 \omega_u) \cap [\sup C_2
    \widetilde{\omega_l},\inf C_3 \tilde{\omega}_u),\] and suppose $I
    \subset \widetilde{I}$
and $P \neq \tilde{P}.$ From (\ref{separationprop2}) we have
$I \subsetneqq \widetilde{I}.$
We also have $\inf C_2 \tilde{\omega}_l < \sup C_2 \omega_l$ since
    otherwise it would follow that $\tilde{S}$ was selected prior to
    $S$ and hence $P \in \widetilde{T}$ which is impossible. From
    (\ref{separationofscales}), we have $\inf C_2 \widetilde{\omega}_l
    \geq \sup C_3 \omega_l$ and so $P$ is in the maximal
    $l$-overlapping tree contained in $T'$ with top data
$(\widetilde{I},\inf C_2 \widetilde{\omega}_l)$.

For each $T' \in \T'$ let $T''$ be the collection of multitiles
$P \in T' \cap \bigcup_{T \in \T} S$ with
$\xi_{T'} \in [\sup C_2 \omega_l,\inf C_3 \omega_u)$ and $I$ maximal
among such multitiles.
Then
\[
\sum_{T' \in \T'} \sum_{\substack{P \in T' \cap \bigcup_{T \in \T} S
    \\ \sup C_2 \omega_l \leq \xi_{T'} < \inf C_3 \omega_u}}
|\<f,\phi_P\>|^2 \leq \sum_{T' \in \T'} \sum_{P'' \in T''}
\sum_{\substack{P \in T' \cap \bigcup_{T \in \T} S \\ \sup C_2
    \omega_l \leq \xi_{T'} < \inf C_3 \omega_u \\ I \subset I''}}
|\<f, \phi_P\>|^2
\]
Considering the discussion in the preceding paragraph, we may apply
the energy bound
to the maximal $l$-overlapping tree contained in $T'$ with top data
$(\widetilde{I},\inf C_2 \widetilde{\omega}_l)$ . Since the $I''$ are
disjoint subintervals of $I_{T'}$ we see
that  the right side in the last display is bounded by
\[
\sum_{T' \in \T'} \sum_{P'' \in T''} 2 \ene^2|I''| \leq \sum_{T' \in \T'} 2 \ene^2 |I_{T'}|.
\]
This completes the proof of  (\ref{pisatree}).
\end{proof}

The proposition below is for use in tandem with Proposition \ref{energyincrementprop}.

\begin{proposition} \label{densityincrementprop}
Let $\P$ be a collection of multitiles and $\den > 0$. Then, there is a collection of trees $\T$ such that
\begin{equation} \label{dtreeintervalbound}
\sum_{T \in \T} |I_T| \lc \den^{-r'} |E|
\end{equation}
and such that
\[
\density\left( \P \setminus \cup_{T \in \T} T\right) \leq \den/2.
\]

\end{proposition}

\begin{proof}
We select trees through an iterative procedure. Suppose that trees $T_j,$
$T^+_j,$ $T^-_j$ have been chosen for $j=1, \ldots, k$.
Let
\[
\P_k = \P \setminus \bigcup_{j=1}^{k} T_j \cup T^+_j \cup T^-_j.
\]
If $\density(\P_k) \leq \den/2$ then we terminate the procedure and set
\[
\T = \{T_1, T^+_1,T^-_1, \ldots, T_k, T^+_k, T^-_k\}.
\]
Otherwise, we may find a nonempty tree $T \subset \P_k$ such that
\begin{equation} \label{treehaslargedensity}
\frac{1}{|I_T|} \int_E \big(1 + \tfrac{|x - c(I_T)}{|I_T|}\big)^{-4} \sum_{k : \xi_{k-1}(x) \in \omega_T} |a_k(x)|^{r'}  \ dx
 > (\den/2)^{r'}.
\end{equation}
Choose $T_{k+1} \subset \P_{k}$ so that $|I_{T_{k+1}}|$ is maximal among all nonempty trees contained in $\P_k$ which satisfy (\ref{treehaslargedensity}), and so that $T_{k+1}$ is the maximal, with respect to inclusion, tree contained in $\P_k$ with top data $(I_{T_{k+1}}, \xi_{T_{k+1}}).$  Let $T^+_{k+1} \subset \P_k$ be the maximal tree contained in $\P_k$ with top data $(I_{T_{k+1}}, \xi_{T_{k+1}} + (C_2 - 1)/(2|I_{T_{k+1}}|))$ and $T^-_{k+1} \subset \P_k$ be the maximal tree contained in $\P_k$ with top data $(I_{T_{k+1}}, \xi_{T_{k+1}} - (C_2 - 1)/(2|I_{T_{k+1}}|)).$
Since each $T_j$ is nonempty and $\P$ is finite, this process will eventually stop.

To prove (\ref{dtreeintervalbound}), it will suffice to verify
\begin{equation} \label{thirddtreeintervalbound}
\sum_{j}|I_{T_j}| \lc \den^{-r'} |E|.
\end{equation}
To this end, we first observe that the tiles $I_{T_j} \times \omega_{T_j}$ are pairwise disjoint. Indeed, suppose that $(I_{T_j} \times \omega_{T_j}) \cap (I_{T_{j'}} \times \omega_{T_{j'}}) \neq \emptyset$ and $j < j'.$ Then, by the first maximality condition, we have $|I_{T_{j}}| \geq |I_{T_{j'}}|$ and so $I_{T_{j'}} \subset I_{T_j}$ and $|\omega_{T_{j}}| \leq |\omega_{T_{j'}}|.$ From the latter inequality, it follows that for every $P \in T_{j'}$, either $\omega_{T_{j}} \subset \omega_m$, $\omega_{T^+_{j}} \subset \omega_m$, or $\omega_{T^-_{j}} \subset \omega_m.$ Thus, $T_{j'} \subset T_j \cup T_j^+ \cup T_j^-$  which contradicts the selection algorithm.

Breaking the integral up into pieces and applying a pigeonhole argument, it follows from (\ref{treehaslargedensity}) that for each $j$ there is a positive integer $\ell_j$ such that
\begin{equation} \label{trunctreehaslargedensity}
|I_{T_j}| \leq C 2^{-3\ell_j} \den^{-r'} \int_{E \cap 2^{\ell_j}I_{T_j}} \,\sum_{k : \xi_{k-1}(x) \in \omega_{T_j}} |a_k(x)|^{r'}  \ dx.
\end{equation}

For each $\ell$ we let $\T^{(\ell)} = \{T_j : \ell_j = \ell\}$ and choose elements of $\T^{(\ell)}$: $T^{(\ell)}_1, T^{(\ell)}_2, \ldots$ and subsets of $\T^{(\ell)}$: $\T^{(\ell)}_1, \T^{(\ell)}_2, \ldots$ as follows. Suppose $T^{(\ell)}_j$  and $\T^{(\ell)}_j$  have been chosen for $j=1, \ldots, k.$ If $\T^{(\ell)} \setminus \bigcup_{j=1}^k \T^{(\ell)}_j$ is empty, then terminate the selection procedure. Otherwise, let $T^{(\ell)}_{k+1}$ be an element of $\T^{(\ell)} \setminus \bigcup_{j=1}^k \T^{(\ell)}_j$ with $|I_{T^{(\ell)}_{k+1}}|$ maximal, and let
\[
\T^{(\ell)}_{k+1} = \{T \in \T^{(\ell)} \setminus \bigcup_{j=1}^k \T^{(\ell)}_j : (2^\ell I_T \times \omega_T) \cap (2^\ell I_{T^{(\ell)}_{k+1}} \times \omega_{T^{(\ell)}_{k+1}}) \neq \emptyset\}.
\]

By construction, $\T^{(\ell)} = \bigcup_j \T^{(\ell)}_j$ and so
\begin{equation} \label{sumtruncIT}
\sum_{T \in \T^{(\ell)}} |I_{T}| \leq \sum_{j} \sum_{T \in \T^{(\ell)}_j} |I_{T}|.
\end{equation}
Using the fact that the tiles $I_{T_j} \times \omega_{T_j}$ are pairwise disjoint, and (twice) the fact that $|I_T| \leq |I_{T^{(\ell)}_j}|$ for every $T \in \T^{(\ell)}_j,$ we see that for each $j$
\[
\sum_{T \in \T^{(\ell)}_j} |I_{T}| \lc 2^{\ell} |I_{T^{(\ell)}_{j}}|.
\]
From (\ref{trunctreehaslargedensity}), we thus see that the right side
of (\ref{sumtruncIT}) is dominated by a constant times
\begin{align*}
&2^{-2\ell} \den^{-r'} \int_E \sum_j \indic{2^\ell I_{T^{(\ell)}_{j}}}(x) \sum_{k : \xi_{k-1}(x) \in \omega_{T^{(\ell)}_{j}}} |a_k(x)|^{r'}  \ dx
\\
&=2^{-2\ell} \den^{-r'} \int_E \sum_{k}|a_k(x)|^{r'}
\# \big\{j: (x,\xi_{k-1}(x))\in
 2^\ell I_{T^{(\ell)}_{j}}\times   \omega_{T^{(\ell)}_{j}} \big\}
\, dx
\\ &\le 2^{-2\ell} \den^{-r'} |E|.
\end{align*}
where we used the disjointness of the  rectangles $2^\ell I_{T^{(\ell)}_{j}} \times \omega_{T^{(l)}_{j}}$, and $\sum_{k = 1}^{K}|a_k(x)|^{r'} \leq 1$.
Summing over $\ell$, we obtain (\ref{thirddtreeintervalbound}).
\end{proof}

\section{The tree estimate} \label{treeestimatessection}
In this section we prove the  basic estimate for
the model operators
in the special case where the collection of multitiles is a tree.
 In what follows we use the notation
$\cV^r\!A f(x)$ for the $r$-variation   of
$k\mapsto A_k f(x)$, for a given  family of operators $A_k$ indexed by
$k\in \bbN$.

An essential tool introduced to harmonic analysis by Bourgain \cite{bourgain89pet}
is
{\it L\'epingle's inequality} for martingales (\cite{lepingle76lvd}).
Consider the martingale of dyadic averages
\[
\expect_k[f](x) = \frac{1}{|I_k(x)|}\int_{I_k(x)} f(y)\ dy
\]
where $I_k(x)$ is the dyadic interval of length $2^k$ containing $x$.
It is a special case of L\'{e}pingle's inequality
that
\begin{equation} \label{vedma}
\|\cV^r\expect_{(\cdot)}[f]\|_{L^p} \leq C_{p,r} \|f\|_{L^p}
\end{equation}
whenever $1 < p < \infty$ and $r > 2$.
Simple proofs
(based on jump inequalities) have been obtained in
 \cite{bourgain89pet}
and \cite{pisier-xu}
(see also \cite{demeter08wmc}, \cite{jones08svj}
for other expositions).
Inequality \eqref{vedma} has been extended to
various   families of convolution operators
(\cite{bourgain89pet},\cite{jkrw},  \cite{cjrw-hilbert},
\cite{jones08svj}).
Let $\psi$ be a Schwartz function on $\rea$ with $\int \psi = 1$, for each $k$ let $\psi_k = 2^{-k}\psi(2^{-k}\cdot)$ and let $A_k f(x)=\psi_k*f$.
Then
\begin{equation} \label{varestpsi}
\|\cV^r\!A f\|_{L^p} \le C'_{p,r} \|f\|_{L^p}, \quad 1<p<\infty, \,\,r>2,
\end{equation}
follows from (\ref{vedma}); the  essential tool  is the square function
estimate
\Be{square-function-estimate}
\Big\|\Big( \sum_{k=-\infty}^{\infty} 
\big|A_k f - \expect_k[f]\big|^2\Big)^{1/2}\Big\|_p, \quad 1<p<\infty.
\Ee
For a proof we refer to \cite{jkrw} or \cite{jones08svj}.

We shall use \eqref{varestpsi} to prove
the following estimate in terms of energy and density.
 The bound will be applied in Section \ref{mainargumentsection} with $q = r'$ and $q=1.$

\begin{proposition} \label{treeestimate} Let $\ene>0$, $0<\den<3$ and
  let  $T$ be a tree with energy bounded above by $\ene$ and density 
bounded above by $\den.$ Then, for each $1 \leq q \leq 2$
\begin{equation} \label{Lqfortree}
\Big\|\sum_{P \in T} \langle f, \phi_P\rangle \phi_P a_P \indic{E}
\Big\|_{L^q} \lc\ene \,\den^{\min(1,r'/q)} |I_T|^{1/q}.
\end{equation}
Furthermore, for $\ell \geq 0$ we have
\begin{equation} \label{treeestimateaway}
\Big\|\sum_{P \in T} \langle f, \phi_P\rangle \phi_P a_P \indic{E}
\Big\|_{L^q(\rea \setminus 2^\ell I_T)} \lc 2^{-\ell(N-10)} \ene \,\den^{\min(1,r'/q)}
|I_T|^{1/q}.
\end{equation}
\end{proposition}

We remark that the  bounds above also hold for $2<q<\infty,$
but the result for this range of exponents is 
not needed for our purposes; it requires 
an additional $L^p$ estimate for 
$\sum_{P \in T} \langle f, \phi_P\rangle \phi_P$.

\begin{proof}
We begin with some  preliminary reduction. Every tree can be split into an
$l$-lacunary tree and an $l$-overlapping tree and it suffices to prove
the asserted estimate for these cases.

If the tree is $l$-overlapping we introduce further decompositions.
By breaking up $T$ into a bounded number of subtrees 
 we may and shall assume without loss of generality that for each $P \in T$,
 \Be{jconsolidation} \xi_T \in \omega_l + j|\omega_l|
\text{  for some integer $j$ with $|j| \leq C_2.$}
\Ee
Moreover we shall assume, for every $l$-overlapping tree $T$,
that either $\xi_T \leq \inf \omega_l$ for every $P \in T$ 
(in which case we refer to $T$ as  $l^-$-overlapping)
or $\xi_T > \inf \omega_l$ for every $P \in T$ 
(in which case we refer to $T$ as  $l^+$-overlapping).
Every $l$-overlapping tree can be split into an
$l^-$-overlapping and an
$l^+$-overlapping tree.
For the remainder of the proof, we assume without loss of generality
that $T$ is either $l$-lacunary or $l^+$-overlapping or
$l^-$-overlapping, 
and that for the last two categories property \eqref{jconsolidation}
is satisfied.

Let $\J$ be the collection of dyadic intervals $J$ which are maximal with respect to the property that $I \not\subset 3J$ for every $P \in T.$

Our first goal is to prove that for each  $J \in \J$
\begin{equation} \label{treeestimatesmallintervals}
\Big\|\sum_{P \in T : |I| < C''|J|} \langle f, \phi_P\rangle \phi_P
a_P \indic{E}
\Big\|_{L^q(J)} \lc \ene \den^{\min(1,r'/q)}
|J|^{1/q} \big(1 + \tfrac{\dis(I_T,J)}{|I_T|}\big)^{-(N-6)}
\end{equation}
where $C'' \geq 1$ is a constant to be determined later; we shall see
that $C''= 8C_1(C_2-1)^{-1}$ is an admissible choice.

 By H\"{o}lder's inequality, we may assume that $q \geq r'$.
Fix $P \in T$ with $|I| \leq C''|J|$. From the energy bound, we have
\begin{equation} \label{tilesmall|I|eq}
\| \langle f, \phi_P\rangle \phi_P a_P \indic{E} \|_{L^q(J)} \lc \ene \big(1 + \tfrac{\dis(I,J)}{|I|}\big)^{-N} \|a_P \indic{E}\|_{L^q(J)}.
\end{equation}
From the $\density$ bound applied to $\approx 1/(C_2 - 1)$ nonempty trees, each with top time interval $I$, we obtain
\[
\frac{1}{|I|} \int_E \Big(1 + \tfrac{|x - c(I)|}{|I|}\Big)^{-4} \sum_{k : \xi_{k-1}(x) \in \omega_l}|a_k(x)|^{r'}\ dx \lc \den^{r'}.
\]
Since $I \not\subset 3J$, it follows that $1 + |x - y|/|I| \leq C (1 + \dis(I,J)/|I|)$ for every $x \in J$ and $y \in I.$ Thus
\[
\|a_P \indic{E}\|_{L^{q}(J)}^{q} \leq \|a_P
\indic{E}\|_{L^{r'}(J)}^{r'} \lc
\big(1 + \tfrac{\dis(I,J)}{|I|}\big)^{4} |I| \den^{r'}
\]
where, above, we use the fact that $|a_P| \leq 1.$
Thus  we can replace  (\ref{tilesmall|I|eq}) by
\[
\big\| \langle f, \phi_P\rangle \phi_P a_P \indic{E}\big \|_{L^q(J)} \lc
 \ene\,\den^{r'/q} |I|^{1/q} \big(1 + \tfrac{\dis(I,J)}{|I|}\big)^{-(N-4)}.
\]
Summing this estimate and using the fact that $T$ is a tree, we have
\[
\Big\|\sum_{P \in T : |I| = 2^{-k}|J|} \langle f, \phi_P\rangle \phi_P
a_P \indic{E}
\Big\|_{L^q(J)} \lc  \ene\,\den^{r'/q}  
(2^{-k}|J|)^{1/q} \big(1 + \tfrac{\dis(I_T,J)}{|I_T|}\big)^{-(N-6)}
\]
and summing over $k$ gives (\ref{treeestimatesmallintervals}).

We now use \eqref{treeestimatesmallintervals} to prove 
\eqref{treeestimateaway}  for $\ell \geq 4.$ 
Indeed, using the maximality of each $J$,
we see that if $\ell \geq 4$ and $J \cap (\rea \setminus 2^\ell I_T) \neq \emptyset$ then $\dis(I_T,J) \geq |J|/2$ and $|J| \geq 2^{\ell-3} |I_T|$. It thus follows from (\ref{treeestimatesmallintervals}) that
\[
\Big\|\sum_{P \in T} \langle f, \phi_P\rangle \phi_P a_P \indic{E} \Big\|_{L^q(J)} \lc (|I_T|/|J|) (\dis(I_T,J)/|J|)^{-2} \ene \den^{\min(1,r'/q)}
|I_T|^{1/q} 2^{-\ell(N-10)}
\]
whenever $J \cap (\rea \setminus 2^\ell I_T) \neq \emptyset.$ Summing over all $J$, we thus obtain (\ref{treeestimateaway}) for $\ell \geq 4.$

It remains to prove
\begin{equation*} 
\Big\|\sum_{P \in T} \langle f, \phi_P\rangle \phi_P a_P \indic{E}
\Big\|_{L^q(16I_T)} \lc \ene \den^{\min(1,r'/q)} |I_T|^{1/q}
\end{equation*}
which, by \eqref{treeestimatesmallintervals}, follows from
\begin{equation} \label{treeestimatelargeintervals} 
\Big\|
\sum_{P \in T\,:\,|I|\ge C''|J|} 
\langle f, \phi_P\rangle \phi_P a_P \indic{E}
\Big\|_{L^q(16I_T)} \lc \ene \den^{\min(1,r'/q)} |I_T|^{1/q}.
\end{equation}
Again, by H\"older's inequality, we may assume that $q \geq r'.$ 
Let $$\Omega_J = \bigcup_{P \in T : |I| \geq C''|J|} \omega_l.$$
The first step  in the proof of \eqref{treeestimatelargeintervals} 
will be to demonstrate
\begin{equation} \label{usedensitybound}
\int_{J \cap E} \sum_{k : \xi_{k-1}(x) \in \Omega_J} |a_k(x)|^{r'}\ dx \lc \den^{r'} |J|.
\end{equation}
By the maximality of $J$ there is a multitile $P^*=(I^*,\om_u^*) \in T$ with 
$I^* \subset 3\tilde{J}$ where $\tilde{J}$ is the dyadic double of
$J$. 
This implies that there is a dyadic interval $J'$ with $|J| \leq |J'|
\leq 4 |J|$ and $\dis(J,J') \leq |J|$ and $I^* \subset J'.$ 
We wish to apply the density bound assumption to a  tree
consisting of the one multitile $P^*$, with top interval $J'$ and
suitable choice of the top frequency. 
We distinguish the cases that $T$ is $l^+$-overlapping, 
$l^-$-overlapping,  or $l$-lacunary.

If $T$ is $l^+$-overlapping then $T' = (\{P^*\},\xi_T,J')$ is a
tree. For every $P \in T$ 
we have $\omega_l\subset
[\xi_T-\frac{C_1}{2|I|},\xi_T+\frac{C_1}{2|I|})$
and thus, with  $|I| \geq C''|J|$, 
$\omega_l\subset [\xi_T - \frac{C_1}{2C''|J|}, \xi_T + \frac{C_1}{2C''|J|})$.
  Thus with the choice of  $C'' \ge \frac{8C_1}{C_2 - 1}$ 
we have $\Omega_J \subset
  \omega_{T'}.$

If $T$ is $l^-$-overlapping then $T' = (\{P^*\},\xi_T + \frac{C_2 -
  1}{4|J|},J')$ is a tree. Using  that $T$ is $l^-$-overlapping, we
see that $\omega_l\subset [\xi_T,\xi_T + \frac{C_2}{2C''|J|})$ for
  every $P \in T$ with $|I| \geq C''|J|$. Thus, by choosing $C'' \geq
   \frac{8 C_2}{C_2 - 1},$ we have $\Omega_J \subset 
\omega_{T'}.$

If $T$ is $l$-lacunary then $T' = (\{P^*\}, \xi_T - \frac{C_2 -
  1}{4|J|}, J')$ is a tree. Using  that $T$ is $l$-lacunary, we see
that 
$\omega_l \subset [\xi_T - \frac{C_1}{2C''|J|},\xi_T)$ for every $P
  \in T$ 
with $|I| \geq C''|J|$. Thus, by choosing $C'' \geq 8 \frac{C_1}{C_2 - 1}$
  as in the first case  we have $\Omega_J \subset \omega_{T'}.$

In any of the three cases, the $\density$ bound gives
\[
\frac{1}{|J'|} \int_E \Big(1 + \frac{|x - c(J')}{|J'|}\Big)^{-4} \sum_{k : \xi_{k-1}(x) \in \omega_{T'}}|a_k(x)|^{r'}\ dx \leq \den^{r'}
\]
and hence (\ref{usedensitybound}).

We now show that if $T$ is $l$-lacunary then 
\eqref{treeestimatelargeintervals} 
follows from (\ref{usedensitybound}). We start by observing that for
each $x$ there is at most one integer $m$ and at most one integer $k$
such that there 
exists a $P \in T$ with $|I| = 2^{m}$, $\xi_{k-1}(x) \in \omega_l$, and
$\xi_k (x) \in \omega_h.$ Indeed,  suppose such a $P$ exists, and $P'
\in T$ with $|I'| \ge |I|.$ If $|I|=|I'|$ the uniqueness of $k$ is
obvious.  Suppose $|I'|>|I|$. 
Since $T$ is $l$-lacunary, we have
$\inf(\omega_l') > \sup(\omega_l)$ by (\ref{separationofscales}), and
so  $\xi_{k-1}(x) < \inf(\omega_l').$ We also have $\xi_k (x)\ge\inf(\om_h)>
\sup(\omega_T) > \sup(\omega_l')$ since $C_2\omega_u \cap \omega_h =
\emptyset$ and $\xi_T \not\in C_2\omega_{l}'.$ It follows that there
is no $k'$ with $\xi_{k'-1}(x) \in \omega_l'$.

Now let 
$a(x) = a_k(x)$ if there exists an $m(x)$ as in the previous
paragraph with $2^{m(x)} \geq C''|J|$, and $a(x) = 0$ otherwise.
We then have
\[\Big\|\sum_{P \in T : |I| \geq C''|J|} \langle f, \phi_P\rangle \phi_P
a_P \indic{E}
\Big\|_{L^q(J)}^q
\leq \int_{J \cap E} \Big( |a(x)| \sum_{P \in T : |I| = 2^{m(x)}}
|\langle f, \phi_P\rangle \phi_P(x)| \Big)^{q}\ dx\,.
\]
From the energy bound (applied to  multitiles) and the bound 
\eqref{wavepacketadaptedbound}
for $|\phi_P|$, the right side
is
bounded by
\[
\int_{J \cap E} \Big[|a(x)| \sum_{\substack{P \in T :\\ |I| = 2^{m(x)}}} \,\ene\,
|I|w_I(x) \Big]^{q}\ dx\,\lc\,
 \int_{J \cap E} \Big[|a(x)| \sum_{\substack{P \in T :\\ |I| = 2^{m(x)}}} \,\ene\,
\big(1 + \tfrac{|x - c(I)|}{|I|}\big)^{-N} \Big]^{q}\ dx.
\]
Noting that $\sum_{P \in T : |I| = 2^{m(x)}} (1 + |x - c(I)|/|I|)^{-N}
\leq C$, we see that the last display  is \[
\lc \ene^q \int_{J \cap E} |a(x)|^q\ dx
\lc  \ene^q \int_{J \cap E} \Big(\sum_{k : \xi_{k-1}(x) \in \Omega_J} |a_k(x)|^{r'}\Big)^{q/r'} \ dx,
\]
by our choice of $a(x)$.
Using (\ref{usedensitybound}), $q\ge r'$,  and the fact that $\sum|a_k(x)|^{r'}
\leq 1$, the right hand side is $\lc \ene^q \den^{r'} |J|.$
We may now sum over those $J\in \J$ which satisfy $J\cap 16
I_T\neq\emptyset$ and 
\eqref{treeestimatelargeintervals} follows for $l$-lacunary trees.

It remains to prove 
\eqref{treeestimatelargeintervals}
 for  the case when $T$ is $l$-overlapping
and satisfies condition \eqref{jconsolidation}.  For each $J\in \J$,
and each $x\in J\cap E$ we have by H\"older's inequality
\begin{align} \label{csvar}
&\sum_{P \in T : |I| \geq C''|J|} \langle f, \phi_P\rangle \phi_P(x)
 a_P (x)= \sum_{k} a_k(x) 
\sum_{\substack{P \in T : |I| \geq C''|J|\\
\xi_{k - 1}(x) \in \omega_l,
    \xi_k(x) \in \omega_h}} \langle f, \phi_P\rangle \phi_P(x)
\\&\,\leq\,
\Big(\sum_{k : \xi_{k-1}(x) \in \Omega_J} |a_k(x)|^{r'}\Big)^{1/r'}
\Big(\sum_{k : \xi_{k-1}(x) \in \Omega_J}  \Big|
\sum_{\substack{P \in T : |I| \geq C''|J|\\ \xi_{k - 1}(x) \in \omega_l,
    \xi_k(x) \in \omega_h}} 
\langle f, \phi_P\rangle \phi_P(x)\Big|^r\Big)^{1/r}.
\notag
\end{align}



Now let $\psi$ be a 
Schwartz function with $\widehat{\psi}(\xi)=1$ for $|\xi| \leq C_1 +
C_3$ and $\widehat{\psi}(\xi) = 0$ for $|\xi| \geq 2C_1$. Define 
$\psi_{\ell} = 2^{-\ell}\psi(2^{-\ell}\cdot),$ and set
\[e_T(x)= e^{2\pi i \xi_T x}.\]
We will show that for any $x$ and any $k$ with  $\xi_{k-1}(x) <
\xi_k(x),$ there exist integers $\ell_1,\ell_2$ depending on $x$ and
$k$,   such that  $2^{\ell_1} \geq |J|$ and 
\begin{equation} \label{schwartzchoice}
\sum_{\substack{P \in T : |I| \geq C''|J|, \\\xi_{k - 1}(x) \in
    \omega_l,\, \xi_k(x) \in \omega_h}} \langle f, \phi_P\rangle \phi_P(x) =
(e_T(\psi_{\ell_1} - \psi_{\ell_2})) * \Big[\sum_{P \in T} \langle f,
\phi_P\rangle \phi_P \Big](x).
\end{equation}
From (\ref{separationofscales}) we have,  for each $\ell$ such that $2^\ell = |I|$ for some multitile $P$,
\[
(e_T \psi_{\ell}) * \sum_{P \in T} \langle f, \phi_P\rangle \phi_P(x) =
\sum_{P \in T : |I| \geq 2^\ell} \langle f, \phi_P\rangle \phi_P(x).
\]
Thus, to prove (\ref{schwartzchoice}) it will suffice to show that there exist integers $\ell_1$ and $\ell_2$ such that
\begin{equation} \label{intervalofscales}
\{P \in T : |I| \geq C''|J|, \xi_{k - 1}(x) \in \omega_l, \xi_k(x) \in \omega_h \} = \{P \in T : 2^{\ell_1} \leq |I| \leq 2^{\ell_2}\}.
\end{equation}
Again using (\ref{separationofscales}), we see that for $P,P' \in T$ with $|I| < |I'|$ we have $\inf \omega_h' < \inf \omega_h$, and if we are in the setting of $\rho$-multitiles where $\rho$ is a $1$-index, we have the stronger inequality $\sup \omega_h' < \inf \omega_h$. Thus, (\ref{intervalofscales}) will follow after finding $\ell_1$ and $\ell_2$ with
\begin{equation} \label{intervalofscales2}
\{P \in T : \xi_{k - 1}(x) \in \omega_l\} = \{P \in T : 2^{\ell_1} \leq |I| \leq 2^{\ell_2}\}.
\end{equation}
Here we use assumption \eqref{jconsolidation}. 
The displayed equation  follows when $|j| > 1$ from the fact that
$\omega_l \cap \omega_l' = \emptyset$ if $P,P' \in T$ and $|I| <
|I'|;$ it follows when $j=0$ from the fact that the intervals
$\{\omega_l : P \in T\}$ are nested. Finally, when $j=\pm 1$ it
follows from the property that if $P,P',P'' \in T$, $|I|, \leq |I'|
\leq |I''|$ and $\omega_l \cap \omega_l'' \neq \emptyset$ then
$\omega_l'' \subset \omega_l' \subset \omega_l.$ Thus we have
established
\eqref{intervalofscales2} and consequently \eqref{intervalofscales}
and \eqref{schwartzchoice}.

Using (\ref{schwartzchoice}), we have
\begin{multline*}
\Big(\sum_{k : \xi_{k-1}(x) \in \Omega_J} \Big| \sum_{\substack{P \in
    T : |I| \geq C''|J|, \\\xi_{k - 1}(x) \in \omega_l, \xi_k(x) \in
    \omega_h}} \langle f, \phi_P\rangle \phi_P(x)\Big|^r\Big)^{1/r}\\
\leq\big \|
(e_T\psi_k )* \sum_{P \in T} \langle
f, \phi_P\rangle \phi_P(x)
\big\|_{V^r_k(\Z^+ + \log_2(|J|))}
\end{multline*}
where the notation refers to the variation norm with respect to the
variable  $k$, restricted to $\{k\in \bbZ: |k|\ge \log_2(|J|)\}$.
For $\log_2(|J|) \leq k_1 < k_2$, we have
\[
(e_T(\psi_{k_1} - \psi_{k_2})) * \sum_{P \in T} \langle f,
  \phi_P\rangle \phi_P \\=
(e_T \psi_{C +\log_2(|J|)}) * (e_T
(\psi_{k_1} - \psi_{k_2})) * \sum_{P \in T} \langle f, \phi_P\rangle \phi_P
\]
and so, for $x \in J$
\begin{multline*}
\big\|
(e_T \psi_k )* \sum_{P \in T} \langle f, \phi_P\rangle \phi_P(x)\big\|_{V^r_k(\Z^+ + \log_2(|J|))}
\\ \lc \sup_{x \in J} \sup_{R \geq |J|} \frac{2}{|R|} \int_{x -
  R}^{x+R} \big\| (e_T
\psi_k )* \sum_{P \in T} \langle f, \phi_P\rangle \phi_P(y)\big\|_{V^r_k(\Z^+ + \log_2(|J|))}\ dy.
\end{multline*}

We now integrate the $q$th power of the expressions in \eqref{csvar} over 
$E\cap J$ and obtain
\begin{align*}
&\Big\|
\sum_{P \in T : |I| \geq C''|J|} 
\langle f, \phi_P\rangle \phi_P
a_P \indic{E}\Big \|_{L^q(J)}^q 
\,\leq\,
\int_{J \cap E}
\Big(\sum_{k : \xi_{k-1}(t) \in \Omega_J}
|a_k(t)|^{r'}dt\Big)^{q/r'}\times\\
&\qquad\sup_{x \in J} \sup_{R \geq |J|} \frac{2}{|R|} \int_{x -
  R}^{x+R} \big\| (e_T
\psi_k )* \sum_{P \in T} \langle f, \phi_P\rangle
\phi_P(y)\big\|_{V^r_k(\Z^+ + \log_2(|J|))}\ dy.
\\
&\leq \den^{r'} \int_J \Big|\mathcal{M}\big[ \| \psi_k
  *
(e_T^{-1}\sum_{P \in T} \langle f, \phi_P\rangle
  \phi_P)\|_{V^r_k}\big](x)\Big|^q\,
dx
\end{align*}
where $\mathcal{M}$ is the Hardy-Littlewood maximal operator. For the
last inequality we have used $q\ge r'$,  $\sum|a_k(x)|^{r'}\le 1$ and 
inequality \eqref{usedensitybound}.
Summing over $J\in\J$ gives
\begin{multline*} \Big\|\sum_{P \in T : |I| \geq C''|J|}  \langle f, \phi_P\rangle \phi_P
a_P \indic{E} \Big\|_{L^q(16I_T)}^q
\\ \lc 
\den^{r'} \big\| \mathcal{M}\big[ \| \psi_k *
(e_T^{-1}\sum_{P \in T} \langle f, \phi_P\rangle \phi_P)\|_{V^r_k} \big](x)\big\|_{L^q_x(16I_T)}^q.
\end{multline*}
Since $q \leq 2$, it follows from H\"older's inequality
that the right side above is
\[
\lc\den^{r'} |I_T|^{(2-q)/2} \big\|
\mathcal{M}\big[ \| \psi_k *
(e_T^{-1}\sum_{P \in T} \langle f, \phi_P\rangle \phi_P)\|_{V^r_k} \big](x)
\big\|_{L^2_x(16I_T)}^q.
\]
Applying the variation estimate (\ref{varestpsi})
with $p=2$ and the $L^2$ estimate for $\mathcal{M}$ one sees that the display above is
\[
\lc\den^{r'} |I_T|^{(2-q)/2} \Big\| \sum_{P \in T} \langle f, \phi_P\rangle \phi_P\Big\|_{L^2}^q.
\]

To finish the proof, it only remains to see that $\| \sum_{P \in T}
\langle f, \phi_P\rangle \phi_P\|_{L^2}^2
\lc \ene^2 |I_T|.$
The left side of this inequality is dominated by
\[ \sum_{P \in T} \sum_{P' \in T}|\< f, \phi_P\>| |\< f, \phi_P'\>| |\<\phi_P,\phi_{P'}\>|
\leq 2 \sum_{P \in T} |\< f, \phi_P\>|^2 \sum_{P' \in T} |\<\phi_P,\phi_{P'}\>|.\]
Since $T$ is an $l$-overlapping tree, we have $\inn{\phi_P}{\phi_{P'}}$
unless $|I| = |I'|$, in which case, we have $|\inn{\phi_P}{\phi_{P'}}|
\lc (1 + \dis(I,I')/|I|)^{-N}.$ Therefore we obtain the estimate
\[ \Big\| \sum_{P \in T}
\langle f, \phi_P\rangle \phi_P\Big\|_{L^2}^2
\lc \sum_{P \in T} |\inn{ f}{ \phi_P}|^2 \leq C \ene^2 |I_T|.
\]
This concludes the proof of
\eqref{treeestimatelargeintervals} and thus the proof of the proposition.
\end{proof}

\section{Two auxiliary estimates}\label{auxiliarysection}
Before we give the argument on how to decompose our operators
into trees with suitable energy and density bounds we need two
auxiliary estimates.

The following proposition   can be found in \cite{thiele06wpa}, p. 12,
or as  a special case of a lemma from \cite{grafakos04bmd}.

\begin{proposition} \label{tilessmallmaxprop}
Let $T$ be an $l$-overlapping tree. Let $\la>0$ and
$\Omega_{\la,D} = \{\ma_D[\indic{F}] > \lambda\}$ where  $\ma_D$ is the maximal dyadic average operator.
 Then
\[
\frac{1}{|I_T|}\sum_{P \in T : I \not \subset
 \Omega_{\la,D}}|\<f,\phi_P\>|^2
\lc \lambda^2\,.
\]
\end{proposition}

The second auxiliary estimate is  the special case of  an estimate
from \cite{grafakos04bmd}, but we  will provide a proof for
convenience.

\begin{proposition} \label{tileslargemaxprop}
Let $\P$ be a finite set of multitiles, and let $\lambda > 0,$ $F
\subset \rea,$ and $|f| \leq \indic{F}.$
Let
$\Omega_\la = \{\ma[\indic{F}] > \lambda\}.$
Then
\begin{equation} \label{tileslargemaxpropc}
\Big\|\sum_{P \in \P : I \subset \Omega} \<f,\phi_P\>\phi_P a_P \Big\|_{L^1(\rea \setminus \Omega)} \lc \frac{|F|}{\lambda^{1/r}}.
\end{equation}
\end{proposition}

\begin{proof}
Fix $\ell$ and let $I_\ell \subset \Omega_\la$ be a dyadic interval satisfying
\begin{equation} \label{2lIOmega}
2^\ell I_\ell \subset \Omega_\la \text{\ and\ } 2^{\ell+1} I_\ell \not \subset \Omega_\la.
\end{equation}
By  Minkowski's inequality, we estimate
\begin{multline*}
\Big(\sum_{P : I = I_\ell} |\<f,\phi_P\>|^2\Big)^{1/2}
\\
\leq \Big(\sum_{P : I = I_\ell} |\<\indic{4I_\ell}
f,\phi_P\>|^2\Big)^{1/2} + \sum_{j=2}^{\infty} \Big(\sum_{P : I =
  I_\ell} |\<\indic{2^{j+1}I_\ell
 \setminus 2^jI_\ell} f,\phi_P\>|^2\Big)^{1/2}.
\end{multline*}
Using orthogonality, the right hand side  is bounded by
\[
 (4|I_\ell|)^{1/2}\|\indic{4I_\ell} f\phi_{P_0}\|_{L^2} + \sum_{j=2}^{\infty} (2^{j+1}|I_\ell|)^{1/2}\|\indic{2^{j+1}I_\ell \setminus 2^jI_\ell} f \phi_{P_0}\|_{L^2}
\]
where $P_0$ is any multitile with $I = I_\ell.$
Applying the bounds (\ref{wavepacketadaptedbound}) and $|f| \leq
\indic{F}$ , we see that the last  display  is
\[
\lc |F \cap 4I_\ell|^{1/2} + \sum_{j=2}^{\infty} C 2^{-j(N-1)} |F \cap 2^{j+1}I_\ell|^{1/2}.
\]
Since $2^{\ell+1}I_\ell \not\subset \Omega_\la,$ we have $|F \cap 2^{j+1}I_\ell|
\leq C 2^{\max(\ell,j)} |I_\ell| \lambda$ for each $j$. Thus, the last
display  is $\lc(2^l \lambda |I_l|)^{1/2}$ and we have proved
\[
\Big(\sum_{P : I = I_\ell} |\<f,\phi_P\>|^2\Big)^{1/2}\lc
(2^\ell \lambda |I_\ell|)^{1/2}.
\]
Similarly,
\[
\sup_{P : I = I_\ell} |\<f,\phi_P\>| \lc 2^\ell \lambda |I_\ell|^{1/2}
\]
and so, by interpolation,
\begin{equation}
\Big(\sum_{P : I = I_\ell} |\<f,\phi_P\>|^r\Big)^{1/r} \lc (2^\ell \lambda)^{1/r'} |I_\ell|^{1/2}
\end{equation}
whenever $2 \leq r \leq \infty.$
For each $\xi$, $I_\ell$ there is at most one $P \in \P$ with
$\xi \in \omega_l$ and $I = I_\ell$. Thus, using the fact that,
for each $x$, $\sum_{k=1}^{K}|a_k(x)|^{r'} \leq 1$, we see that
\[
\Big\|\sum_{P \in \P : I = I_\ell} \<f,\phi_P\>\phi_P a_P \Big\|_{L^1(\rea \setminus \Omega_\la)} \lc  (2^\ell \lambda)^{1/r'} |I_\ell|^{1/2} \|\phi_{P_0}\|_{L^1(\rea \setminus \Omega_\la)}
\]
where $P_0$ is any multitile with $I_0 = I_\ell.$ Using the fact that
 $2^\ell I_\ell \subset \Omega_\la$, it follows that the right side above is
\[
\lc 2^{-\ell(N-2)} \lambda^{1/r'} |I_\ell|.
\]

For $\ell \geq 0$ let $\mathcal{I}_\ell$ be the set of all dyadic intervals satisfying (\ref{2lIOmega}). If $I \subset \mathcal{I}_\ell$ then for each $j > 0$ there are at most 2 intervals $I' \in \mathcal{I}_\ell$ with $I' \subset I$ and $|I'| = 2^{-j}|I|.$ By considering the collection of maximal dyadic intervals in $\mathcal{I}_\ell,$ one sees that
\[
\sum_{I \in \mathcal{I}_\ell}|I| \lc |\Omega_\la|
\]
Thus,
\[
\|\sum_{P \in \P : I \in \mathcal{I}_\ell} \<f,\phi_P\>\phi_P a_P \|_{L^1(\rea \setminus \Omega_\la)} \lc 2^{-\ell(N-2)} \lambda^{1/r'} |\Omega_\la|.
\]
Summing over $\ell$ and applying the weak-type 1-1 estimate for $\ma$ then gives (\ref{tileslargemaxpropc}).
\end{proof}

\section{Conclusion of the proof} \label{mainargumentsection}
Let $r>2$ and $r'\le p< \frac{2r}{r-2}$.
We shall conclude the proof of
\eqref{halfgbound}, with
 $1 \leq |E| \leq 2$,
 and thus of Theorem \ref{maintheorem}.
It will then suffice, by Chebyshev's inequality,  to show
\begin{equation} \label{L1except}
\int_{E\setminus G} | S^{\rho}[f](x)|\, dx  \leq C |F|^{1/p}
\end{equation}
for any  measurable set $E$ with $1\le |E|\le 2$, $|f|\le \indic{F}$
and  some exceptional
set
$G=G(E,F)$ with $|G| \leq 1/4.$

We shall repeatedly apply Propositions \ref{energyincrementprop} and
\ref{densityincrementprop}. 
By Lemma \ref{univ-den-ene-bound} 
the density of  $\P$ (with respect to
the set $E$ above) and the
energy (with respect to $f$) are  bounded by a universal constant $C$.

We distinguish between the cases $|F| > 1$ and $|F|\le 1$ and
first consider the case when $|F| > 1$.
Repeatedly applying Propositions \ref{energyincrementprop} and \ref{densityincrementprop}
we write $\P$ as the disjoint union
\[
\P = \bigcup_{j \geq 0} \bigcup_{T \in \T_j} T
\]
where each $\T_j$ is a collection of trees $T$ each of which have energy bounded by $C 2^{-j/2}|F|^{1/2}$, density bounded by $C 2^{-{j/r'}},$ and satisfy
\[
\sum_{T \in \T_j} |I_T| \lc 2^j.
\]
For each $j$ we apply Proposition \ref{energyincrementprop} again, this time using (\ref{etreebmobound}) and (\ref{pisatree}) to write
\[
\bigcup_{T \in \T_j}T = \bigcup_{k \geq 0} \bigcup_{T \in \T_{j,k}} T
\]
where each tree $T \in \T_{j,k}$ has energy bounded by $C 2^{-(j+k)/2}|F|^{1/2}$, density bounded by $C 2^{-{j/r'}},$ and satisfies
\begin{equation} \label{L1jktree}
\sum_{T \in \T_{j,k}} |I_T| \lc 2^j.
\end{equation}
Moreover, for every $\ell \geq 0$
\begin{equation} \label{BMOjktree}
\Big\|\sum_{T \in \T_{j,k}} \indic{2^\ell I_T}\Big\|_{BMO} \lc 2^{2\ell} 2^{j+k} |F|^{-1}.
\end{equation}
Inequality \eqref {L1jktree} implies $\|\sum_{T \in \T_{j,k}} \indic{2^\ell
  I_T}\|_{L^1} \lc 2^{\ell+j}$, and we may interpolate
the $L^1$ and the $BMO$ bound. Here we use a standard
 technique involving the sharp maximal function from
\S5 in \cite{fs-hardy}. It  follows that for $1 \leq q < \infty$
\[
\Big\|\sum_{T \in \T_{j,k}} \indic{2^\ell I_T}\Big\|_{L^q} \lc 2^{j+k+2\ell} |F|^{-1/q'}.
\]
Let $\epsilon > 0$ be small and $C' > 0$ be large, depending on $p,q,r$. For each $j,k,l$ define
\[
G_{j,k,\ell} = \Big\{x: \sum_{T \in \T_{j,k}} \indic{2^\ell I_T} \geq C' |F|^{-1/q'} 2^{(1 + \epsilon)(j + k + 2\ell)}\Big\}.
\]
By Chebyshev's inequality, we have
\[|G_{j,k,\ell}| \leq c' 2^{-\epsilon(j + k + 2\ell)},\] so setting $G = \bigcup_{j,k,\ell \geq 0}G_{j,k,\ell}$ we have $|G| \leq 1/4.$

Applying Minkowski's inequality gives
\begin{multline*}
\Big\|\indic{E} \sum_{P \in \P} \<f,\phi_P\>\phi_P a_P
\Big\|_{L^{1}(\rea \setminus G)}
\\ \leq \sum_{j,k \geq  0} \Big(\big\|\indic{E} \sum_{T \in \T_{j,k}} \indic{I_T} \sum_{P \in T} \<f,\phi_P\>\phi_P a_P
\big\|_{L^{1}(\rea \setminus G_{j,k,0})}
\\  + \sum_{\ell \geq 1}
\big\|\indic{E} \sum_{T \in \T_{j,k}} \indic{2^{\ell}I_T \setminus
  2^{\ell-1}I_T} \sum_{P \in T} \<f,\phi_P\>\phi_P a_P
\big\|_{L^{1}(\rea \setminus G_{j,k,\ell})}
\Big).
\end{multline*}
From H\"{o}lder's inequality, Fubini's theorem, and the definition of $G_{j,k,l}$, it follows that the right side above is $\lc(S_1 + S_2)$ where
\[
S_1 =  \sum_{\substack{j,k \geq 0}} |F|^{-1/(q'r)} 2^{(1 + \epsilon)(j + k)/r} \Big(\sum_{T \in \T_{j,k}} \big\|\indic{E} \sum_{P \in T} \<f,\phi_P\>\phi_P a_P\big\|_{L^{r'}(\rea)}^{r'}\Big)^{1/r'}
\]
and
\[
S_2 =  \sum_{\substack{j,k \geq 0 \\ \ell \geq 1}} |F|^{-1/(q'r)} 2^{(1 + \epsilon)(j + k + 2\ell)/r} \Big(\sum_{T \in \T_{j,k}}
\big\|\indic{E} \sum_{P \in T} \<f,\phi_P\>\phi_P a_P\big\|_{L^{r'}(\rea \setminus 2^{\ell-1}I_T)}^{r'}\Big)^{1/r'}
\]

Applying Proposition \ref{treeestimate} with the energy and density bounds for trees $T \in \T_{j,k}$, we see that
\begin{align*}
S_2 &\lc \sum_{\substack{j,k \geq 0 \\ \ell \geq 1}} |F|^{-1/(q'r)} 2^{(1 + \epsilon)(j + k + 2\ell)/r} 2^{-\ell(N-10)} 2^{-(j+k)/2}|F|^{1/2}2^{-j/r'} \Big(\sum_{T \in \T_{j,k}} |I_T|\Big)^{1/r'} \\
&\lc \sum_{\substack{j,k \geq 0 \\ \ell \geq 1}} 
2^{(j+k)((1 + \epsilon)(2/r) - 1)/2} 2^{-\ell(N-14)} |F|^{1/2 - 1/(q'r)}\,.
\end{align*}
Choosing $\epsilon$ small enough and $q$ large enough so that
\[(1 + \epsilon)\frac 2r - 1 < 0 \text{ and } \frac 12 - \frac{1}{q'r}
< \frac 1p\]
 we have $S_2 \lc|F|^{1/p}.$ We similarly obtain $S_1 \lc |F|^{1/p},$ thus giving (\ref{L1except}).

We will finish by proving (\ref{L1except}) for $|F| \leq 1$. Here, we
let 
\[G = \{x:\ma[\indic{F}](x) > C''|F|\}\] where $\ma$ is the Hardy-Littlewood maximal operator and $C''$ is chosen large enough so that the weak-type 1-1 estimate for $\ma$ guarantees $|G| \leq 1/4$.

By Proposition \ref{tileslargemaxprop} and the fact that $p \geq r'$, it will remain to show that
\begin{equation} \label{tilesnotlargemax}
\Big\|\indic{E} \sum_{P \in \P'} \<f,\phi_P\>\phi_P a_P \Big\|_{L^1(\rea \setminus G)} \lc |F|^{1/p.}
\end{equation}
where $\P' = \{P \in \P : I \not\subset G\}.$

Finally, it follows from Proposition \ref{tilessmallmaxprop}
that the energy of $\P'$ is bounded above by $C|F|$.
Repeatedly applying Propositions \ref{energyincrementprop} and \ref{densityincrementprop}
we write $\P'$ as the disjoint union
\[
\P' = \bigcup_{j \geq 0} \bigcup_{T \in \T_j} T
\]
where each $\T_j$ is a collection of trees $T$ each of which have energy bounded by $C_\circ 2^{-j/2}|F|^{1/2}$, density bounded by $C_\circ 2^{-{j/r'}},$ and satisfy
\[
\sum_{T \in \T_j} |I_T| \lc 2^j.
\]
We then have
\[
\Big\|\indic{E} \sum_{P \in \P'} \<f,\phi_P\>\phi_P a_P \Big\|_{L^1} \leq \sum_{j \geq 0} \sum_{T \in \T_j} \Big\|\indic{E} \sum_{P \in T} \<f,\phi_P\>\phi_P a_P\Big\|_{L^1}.
\]
Applying Proposition \ref{treeestimate}, we see that the right side above is
\[
\lc \sum_{j \geq 0} \sum_{T \in \T_j} \min(2^{-j/2}|F|^{1/2},|F|) 2^{-j/r'}|I_T| \lc \sum_{j \geq 0} 2^{j/r} \min(2^{-j/2}|F|^{1/2},|F|).
\]
Summing over $j$, we see that the right side above is $\lc |F|^{1/r'}.$ This finishes the proof, since $p \geq r'.$

\begin{appendices}
\section{Transference}\label{transference}
In this section we show how to obtain Theorem \ref{maintheoremper} from
Theorem \ref{maintheorem}. We employ arguments from chapter  VII in the monograph by Stein and Weiss
\cite{steinweiss} in their proof of De Leeuw's transference result
(\cite{leeuw}).  The following limiting relation is used:
\begin{lemma} \label{stwe} (\cite{steinweiss}, p. 261).
Let $m\in L^\infty(\bbR^d)$  with the property that every $k\in
\bbZ^d$  is a
Lebesgue point  of $m$. Define a convolution operator $T$ on $L^2(\Bbb R^d)$
by the Fourier transform identity $\widehat {Tf}(\xi)=m(\xi) \widehat f(\xi)$
and a convolution
operator on $L^2(\bbT^d)$ by the relation $[\widetilde T f]_k\!\hat{}  =
m(k)  f_k\!\hat{}$
for the Fourier coefficients. Let, for $x\in \bbR$,
\[w(x)= e^{-\pi x^2} \text{ and } w_R(x)= R^{-d} w(x/R).\]
Then, for all trigonometric  polynomials $P$ and $Q$ (extended as
$1$-periodic functions in every variable) we have
\Be {steinweisslimit} \int_{[0,1]^d}  \widetilde T [P](x) Q(x)\, dx=
\lim_{R\to \infty}\int_{\bbR^d}  T[Pw_{R/p'}](x) Q(x) w_{R/p}(x)\, dx.
\Ee
\end{lemma}

We also need  the following elementary fact  on Lorentz spaces.
\begin{lemma}\label{Lorentzmultiplicationbd} Let $f\in L^{p,q}(\bbT^d)$ and extend $f$ to a function
  $f^\per$ on $\bbR^d$ which is $1$-periodic in every variable.
Let $L>d/p$ and let  $w$ be a measurable function satisfying $|w(x)|\le (1+|x|)^{-L}$. Let $w_R(x)= R^{-d}w(R^{-1}x)$.
Then
\[ \sup_{R\ge 1}\big\|f^\per w_R\big\|_{L^{p,q}(\bbR^d)} \le C_{p,q}
\|f\|_{L^{p,q}(\bbT^d)}
\]
\end{lemma}
\begin{proof} We first assume $p=q$.
Let $Q_0:=[-\frac12,\frac 12]^d$. If $N\in \bbN$ and $N\le R\le N+1$
 then $w_R(x) \approx N^{-d}(1+|n|/N)^{-Lp}$ for $x\in n+Q_0$ and by the
 periodicity we can estimate $\|f^\per w_R\|_p^p$ by $C \sum_{n\in \bbZ^d} N^{-d}(1+|n|/N)^{-Lp}
\|f\|_{L^p(Q_0)}^p$ which is $ \lc  \|f\|_{L^p(Q_0)}^p$ since
$Lp>d$.
For fixed $L$ we apply real interpolation in the range $p<L/d$ and
obtain the Lorentz space result.
\end{proof}

\begin{proof}[Proof that Theorem \ref{maintheorem} implies
 Theorem \ref{maintheoremper}]
We shall assume that $1<p<\infty$, $1\le q_1<\infty$,
and  $q_1\le q_2\le\infty$ and prove that the
$L^{p,q_1}(\bbR) \to L^{p,q_2}(\bbR; V^r)$ for the partial sum operator $\cS$ on the real line implies the  corresponding  result on the torus, i.e.
\Be{lorentztorus}
\Big\| \sup_K \sup_{0\le n_1 \le \dots \le n_K}
\Big(\sum_{i=1}^{K-1}|S_{n_{i+1}}f - S_{n_i} f
|^r\Big)^{1/r}\Big\|_{L^{p,q_2}(\bbT)}\lc \|f\|_{L^{p,q_2}(\bbT)}
\Ee

By two applications of the monotone convergence theorem it suffices to show for fixed $K\in \bbN$ with $K\ge 2$, and fixed  $M\in \bbN$ that
\Be{lorentztorusKM}
\Big\|  \sup_{0\le n_1 \le \dots \le n_K\le M}
\Big(\sum_{i=1}^{K-1}|S_{n_{i+1}}f - S_{n_i} f
|^r\Big)^{1/r}\Big\|_{L^{p,q_2}(\bbT)}\le C \|f\|_{L^{p,q_1}(\bbT)}
 \Ee
where $C$ does not depend on $M$ and $K$.

For $\vec n=(n_1,\dots, n_K)\in \bbN_0^K$,   $1\le i\le K-1$
define   $\cT f(x,\vec n,i)=  S_{n_{i+1}}f(x) - S_{n_i} f(x)$ if $n_1\le \dots
\le n_K$ and $\cT f(x,\vec n,i)= 0$ otherwise.
Then the inequality \eqref{lorentztorusKM} just says that $\cT$ is bounded
from $L^{p,q_1}$ to $L^{p,q_2}(\ell^\infty(\ell^r))$ where the $\ell^\infty$ norm is taken for functions on the finite set $\{1,\dots, M\}^K$ and
the $\ell^r$ norm is for functions on $\{1,\dots, K-1\}$.
By duality \eqref{lorentztorusKM} follows from the
$L^{p', q_2'}(\ell^1(\ell^{r'}))\to L^{p', q_1'}$
 inequality for the adjoint operator $\cT^*$, i.e.
from the inequality
\begin{multline}\label{lorentztorusKMadj}
\Big| \int_0^1   \sum_{0\le n_1 \le \dots \le n_K\le M}  \sum_{i=1}^{K-1}
\big[S_{n_{i+1}}f_{\vec n, i}(x) - S_{n_i} f_{\vec n, i}(x)
 \big] Q(x)\,
dx \Big|\\
\lc
\Big\|\sum_{\vec n}\Big(\sum_{i=1}^{K-1} |f_{\vec n, i}|^{r'}\Big)^{1/r'}
\Big\|_{L^{p',q_2'}(\bbT)} \,\|Q\|_{L^{p,q_1}(\bbT)}.
\end{multline}

We fix an irrational number $\la$ in $(0,1)$, say $\la=1/\sqrt 2$.
We then define ``partial sum operators'' for Fourier integrals
by $\widehat {\fS_t f} (\xi)= \chi_{[-\la,\la]}(\xi/t) \widehat f(\xi)$ and
a corresponding partial sum operator $\widetilde \fS_t$
 on Fourier series by letting  the $k$th Fourier coefficient of
$\widetilde \fS_t f$ be  equal to $\chi_{[-\la,\la]}(k/t)
\widehat f_k$.
We define a function $\nu: \bbN_0\to \bbN_0$ as follows: set
$\nu(0)=0$ and
 for $n>0$ let $\nu(n)$ be the smallest positive integer $\nu$
for which $\la\nu>n$.
Notice that then
\Be{SversusfS}
S_{n_{i+1}}f_{\vec n, i}(x) -
S_{n_{i}}f_{\vec n, i}(x) =
\widetilde {\fS}_{\nu(n_{i+1})} f_{\vec n, i} - \widetilde{\fS}_{\nu(n_{i})} f_{\vec n, i} .
\Ee



Now in order to prove  \eqref{lorentztorusKMadj} it  clearly
suffices to
verify it for the case that
the function $f_{\vec n,i}$ and $Q$ are trigonometric polynomials.
The multipliers corresponding to $\fS_t$ are continuous at every
integer.  Thus by \eqref{steinweisslimit} (applied with $d=1$) and
\eqref{SversusfS}
we see that
\eqref{lorentztorusKMadj} is implied by
\begin{multline}\label{lorentztorusKMadjw}
\Big| \sum_{0\le n_1 \le \dots \le n_K\le M}  \sum_{i=1}^{K-1}
\int_0^1
\big[\fS_{\nu(n_{i+1})}(f_{\vec n, i}w_{R,p'}) -
\fS_{\nu(n_{i})}(f_{\vec n, i}w_{R,p'})\big]
Q w_{R/p} \,dx\Big|
\\
\lc
\Big\|\sum_{\vec n}\Big(\sum_{i=1}^{K-1} |f_{\vec n, i}|^{r'}\Big)^{1/r'}
\Big\|_{L^{p',q_2'}(\bbT)} \,\|Q\|_{L^{p,q_1}(\bbT)}.
\end{multline}
for sufficiently large $R$.

Now notice that $\fS_t= \cS_{\la t}- \cS_{-\la t}$ so that the assumed
$L^{p,q_1}(\bbR)\to L^{p,q_2}(V^r,\bbR)$ boundedness
 for the family $\{\cS_t\}$
implies the analogous statement  for the family $\{\fS_t\}$. We run
the duality argument in the reverse direction
 (now for functions defined on $\bbR$)
and  deduce
\begin{multline*}
\label{lorentzdualR}
\Big| \int_{\bbR}   \sum_{0\le n_1 \le \dots \le n_K\le M}  \sum_{i=1}^{K-1}
\Big(\fS_{\nu(n_{i+1})}[f_{\vec n, i} w_{R/p'}] - \fS_{\nu(n_i)}
[f_{\vec n, i} w_{R/p'}]
 \Big) Q w_{R/p} \,
dx \Big|\\
\lc
\Big\|\sum_{\vec n}\Big(\sum_{i=1}^{K-1} |f_{\vec n,
  i}|^{r'}\Big)^{1/r'} w_{R/p'}
\Big\|_{L^{p',q_2'}(\bbR)} \,\|Q w_{R/p}\|_{L^{p,q_1}(\bbR)}.
\end{multline*}
By Lemma \ref{Lorentzmultiplicationbd} the right hand side of this
inequality is for $R\ge \max\{p,p'\}$  bounded by
the right hand side of
\eqref{lorentztorusKMadjw}. Thus we have established inequality
\eqref{lorentztorusKMadjw} and this concludes the proof.
\end{proof}

\section{A variational Menshov-Paley-Zygmund theorem} \label{vmpzsection}

For $\xi,x \in \rea$ let
\[\CC[f](\xi,x) = \int_{-\infty}^x e^{-2\pi i \xi x'}f(x')\ dx'.\] Menshov, Paley, and Zygmund extended the Hausdorff-Young inequality by proving a version of the bound
\begin{equation} \label{mpztheorem}
\|\CC[f]\|_{L^{p'}_\xi(L^{\infty}_x)} \leq C_{p} \|f\|_{L^{p}(\rea)}
\end{equation}
for $1 \leq p < 2.$ The bound at $p=2$ is a special case of the much more difficult maximal inequality for the partial sum operator
of
proved by
Carleson and Hunt.
Interpolating  Theorem \ref{maintheorem} at $p=2$ with a trivial estimate at $p=1$, one obtains the following stregthened version of (\ref{mpztheorem})
\begin{equation} \label{vmpztheorem}
\|\CC[f]\|_{L^{p'}_\xi(V^r_x)} \leq C_{p,r} \|f\|_{L^{p}(\rea)}
\end{equation}
for $1 \leq p \leq 2$ and $r > p.$ It follows from the same arguments given in Section \ref{counterexamplesection} that this range of $r$ is the best possible. Our interest in this variational bound primarily stems from the fact, which will be proven in Appendix \ref{vliegroupsection}, that it may be transferred, when $r < 2,$ to give a corresponding estimate for certain nonlinear Fourier summation operators. The purpose of the present appendix is to give an easier alternate proof of (\ref{vmpztheorem}) when $p < 2$.
Note that Pisier and Xu
\cite{pisier-xu}
 have proved closely related $L^p\to L^2(V^p)$ inequalities for
 orthonormal systems of (not necessarily bounded) 
 functions on an arbitrary measure space.

A now-famous lemma of Christ and Kiselev \cite{christ01mfa} asserts that if an integral operator
$$ Tf(x) = \int_\R K(x,y) f(y)\ dy$$
is bounded from $L^p(\R)$ to $L^q(X)$ for some measure space $X$ and some $q > p$, thus
$$ \| Tf \|_{L^q(X)} \leq A \| f \|_{L^p(\R)},$$
then automatically the maximal function
$$ T_* f(x) = \sup_{N\in\R}\Big|\int_{y < N} K(x,y) f(y)\ dy\Big|$$
is also bounded from $L^p(\R)$ to $L^q(X)$, with a slightly larger constant.  Another way to phrase this is
as follows.  If we define the partial integrals
$$ T_\leq f(x,N) = \int_{y < N} K(x,y) f(y)\ dy$$
then we have
\begin{equation} \label{cklemmabound}
\| T_\leq f \|_{L^q_x(L^\infty_N)} \leq C_{p,q} A \| f \|_{L^p(\R)}.
\end{equation}
As was observed by Christ and Kiselev, this may be applied in conjunction with the Hausdorff-Young inequality to obtain (\ref{mpztheorem}) for $p < 2.$

The $L^\infty_N$ norm can also be interpreted as the $V^\infty_N$ norm, and we will now see that $V^\infty$ can be replaced by $V^r$ for $r > p,$ thus giving (\ref{vmpztheorem}) from the Hausdorff-Young inequality.
\begin{lemma}  Under the same assumptions, we have
$$ \| T_\leq f \|_{L^q_x(V^r_N)} \leq C_{p,q,r} A \| f \|_{L^p(\R)}$$
for any $r > p$.
\end{lemma}

\begin{proof}
This follows by an adaption of the argument by Christ and Kiselev,
or by the following argument.
Without loss of generality we may take $r < q$, in particular $r < \infty$.
We use a bootstrap argument.  Let us make the \emph{a priori} assumption that
\begin{equation}\label{pre-bootstrap}
 \| T_\leq f \|_{L^q_x(V^r_N)} \leq BA \| f \|_{L^p(\R)}
\end{equation}
for \emph{some} constant $0 < B < \infty$; this can be accomplished for instance by truncating the kernel $K$
appropriately.  We will show that this a priori bound automatically implies the bound
\begin{equation}\label{post-bootstrap}
 \| T_\leq f \|_{L^q_x(V^r_N)} \leq (2^{1/r - 1/p} BA + C_{p,q,r}A) \| f \|_{L^p(\R)}
\end{equation}
for some $C_{p,q,r} > 0$.  This implies that the best bound $B$ in the above inequality will necessarily
obey the inequality
$$ B \leq 2^{1/r - 1/p} B + C_{p,q,r};$$
since $r > p$, this implies $B \leq C'_{p,q,r}$ for some finite $C'_{p,q,r}$, and the claim follows.

It remains to deduce \eqref{post-bootstrap} from \eqref{pre-bootstrap}.  Fix $f$; we may normalize
$\|f\|_{L^p(\R)} = 1$.  We find a partition point $N_0$ in the
real line which halves the $L^p$ norm of $f$:
$$ \int_{-\infty}^{N_0} |f(y)|^p\ dy = \int_{N_0}^{+\infty} |f(y)|^p\ dy = \frac{1}{2}.$$
Write $f_-(y) = f(y) \indic{(-\infty,N_0]}(y)$ and $f_+(y) = f(y) \indic{[N_0,+\infty)}(y)$, thus
$\|f_-\|_{L^p(\R)} = 2^{-1/p}$ and $\|f_+\|_{L^p(\R)} = 2^{-1/p}$.  We observe that
$$ T_\leq f(x,N) =
\left\{ \begin{array}{ll}
T_\leq f_-(x,N) & \hbox{ when } N \leq N_0\\
Tf_-(x) + T^\leq f_+(x,N) & \hbox{ when } N > N_0
\end{array}\right.
$$
Furthermore, $T_\leq f_-(x,\cdot)$ and $T_\leq f_+(x,\cdot)$ are bounded in $L^\infty$ norm by $O(T_* f(x))$.  Thus we have
$$ \| T_\leq f(x,\cdot) \|_{V^r_N} \leq (\| T_\leq f_-(x,\cdot) \|_{V^r_N}^r + \| T_\leq f_+(x,\cdot) \|_{V^r_N}^r)^{1/r}
+ O( T_* f(x) ).$$
(The $O(T_* f(x))$ error comes because the partition used to define $\| T_\leq f(x,\cdot) \|_{V^r_N}$ may have
one interval which straddles $N_0$).  We take $L^q$ norms of both sides to obtain
$$ \| T_\leq f \|_{L^q_x V^r_N} \leq \| (\| T_\leq f_-(x,\cdot) \|_{V^r_N}^r + \| T_\leq f_-(x,\cdot) \|_{V^r_N}^r)^{1/r} \|_{L^q_x}
+ O( \| T_* f \|_{L^q_x} ).$$
The error term is at most $C_{p,q} A$ by the ordinary Christ-Kiselev lemma.  For the main term, we take
advantage of the fact that $r < q$ to interchange the $l^r$ and $L^q$ norms, thus obtaining
$$ \| T_\leq f \|_{L^q_x V^r_N} \leq (\| T_\leq  f_- \|_{L^q_x V^r_N}^r + \| T_\leq f_+ \|_{L^q_x V^r_N}^r)^{1/r}
+ O( C_{p,q} A ).$$
By inductive hypothesis we thus have
$$ \| T_\leq f \|_{L^q_x V^r_N} \leq ((2^{-1/p} BA)^r + (2^{-1/p} BA)^r)^{1/r}
+ O( C_{p,q} A ),$$
and the claim follows.

\end{proof}

\section{Variation norms on Lie groups} \label{vliegroupsection}
In this appendix, we will show that certain $r$-variation norms for curves on Lie groups can be controlled by
the corresponding variation norms of their ``traces'' on the Lie algebra as long as $r < 2$.  This follows from work of Terry Lyons \cite{lyons94de}; we
present a self contained proof in this appendix.
Combining this fact with the variational Menshov-Paley-Zygmund theorem of
Appendix \ref{vmpzsection}, we rederive the Christ-Kiselev theorem on the
pointwise convergence
of the nonlinear Fourier summation operator for $L^p(\R)$ functions, $1 \leq p < 2$.

Let $G$ be a connected finite-dimensional Lie group with Lie
algebra ${\mathfrak g}$.  We give ${\mathfrak g}$ any norm $\|\cdot\|_{\mathfrak g}$, and push forward this norm using left multiplication by the Lie group
to define a norm $\| x \|_{T_g G} = \| g^{-1}x \|_{\mathfrak g}$ on each tangent space $T_g G$ of
the group. Observe that this norm structure is preserved under left
 group multiplication.

We can now define the \emph{length} $|\gamma|$ of a continuously differentiable
path $\gamma: [a,b] \to G$ by the usual formula
$$ |\gamma| = \int_a^b \| \gamma'(t) \|_{T_{\gamma(t)} G}\ dt.$$
Observe that this notion of length is invariant under left group multiplication, and also under reparameterization of the
path $\gamma$.

From this notion of length, we can define a metric $d(g,g')$ on $G$
as
$$ d(g,g') = \inf_{\gamma: \gamma(a) = g, \gamma(b) = g'} |\gamma|$$
where $\gamma$ ranges over all differentiable paths from $g$ to $g'$.
It is easy to see that this does indeed give a metric on $G$.

Integral curves of left invariant vectorfields need not
be geodesic for this metric \cite{pripoae}, but the length of a 
short segment of such an integral curve is within a quadratically small 
error of the distance between the two endpoints. This is the content
of the following lemma:

\begin{lemma}\label{terryfix}  If $x \in {\mathfrak g}$ is such that $\|x\|_{\mathfrak
g} \leq \eps$ for some sufficiently small $\eps$, then $d( 1, \exp(x)
) = \|x\|_{\mathfrak g} + O( \|x\|_{\mathfrak g}^2 )$.
\end{lemma}

\begin{proof} By considering the exponential curve $\gamma: [0,1] \to
G$ defined by $\gamma(t) := \exp( t x )$ we obtain the upper bound $d(
1, \exp(x) ) \leq \|x\|_{\mathfrak g}$.  Now consider any competitor
curve $\tilde \gamma: [0,1] \to G$ from $1$ to $\exp(x)$ which has
shorter length than $\|x\|_{\mathfrak g}$.  We write $\tilde \gamma(t)
= \exp( f(t) )$ for some smooth curve $f: [0,1] \to {\mathfrak g}$
from $1$ to $x$; this is well-defined if $\eps$ is small enough.

There are two cases.  First suppose that $f$ stays inside the ball $\{
y: \|y\|_{\mathfrak g} \leq 2 \|x\|_{\mathfrak g} \}$.  Then from
Taylor expansion we see that

$$ \| \tilde \gamma'(t) \|_{T_{\gamma(t)} G} = (1 + O(
\|x\|_{\mathfrak g} ) ) \| f'(t) \|_{\mathfrak g}$$

and hence

$$ |\tilde \gamma| = (1 + O( \|x\|_{\mathfrak g} ) ) \int_0^1
\|f'(t)\|_{\mathfrak g}\ dt.$$

But from Minkowski's inequality one has $\int_0^1 \|f'(t)\|_{\mathfrak
g}\ dt \geq \|x\|_{\mathfrak g}$, and the claim follows.

Now suppose instead that $f$ leaves this ball.  Let $0 < t_0 < 1$ be
the first time at which this occurs.  Then the above argument gives

$$ |\tilde \gamma| \geq (1 + O( \|x\|_{\mathfrak g} ) ) \int_0^{t_0}
\|f'(t)\|_{\mathfrak g}\ dt.$$

By Minkowski's inequality one has $\int_0^{t_0} \|f'(t)\|_{\mathfrak
g}\ dt \geq 2 \|x\|_{\mathfrak g}$, and this gives a contradiction to
$|\tilde \gamma| \leq \|x\|_{\mathfrak g} \leq \eps$ if $\eps$ is
sufficiently small.  The claim follows.
\end{proof}

Given any continuous path $\gamma: [a,b] \to G$ and $1 \leq r < \infty$,
we define the \emph{$r$-variation} $\| \gamma \|_{V^r}$ of $\gamma$
to be the quantity
$$ \|\gamma \|_{V^r} = \sup_{a = t_0 < t_1 < \ldots < t_n = b}
\Big(\sum_{j=0}^{n-1} d(\gamma(t_{j+1}), \gamma(t_j))^r\Big)^{1/r}$$
where the infimum ranges over all partitions of $[a,b]$ by finitely
many times $a=t_0$, $t_1,$ $\ldots,$ $t_n = b$.  We can extend this to
the $r=\infty$ case in the usual manner as
$$ \|\gamma \|_{V^\infty} = \sup_{a = t_0 < t_1 < \ldots < t_n = b}\,\,
\sup_{0 \leq j \leq n-1} \,d(\gamma(t_{j+1}), \gamma(t_j)),$$
and indeed it is clear that the $V^\infty$ norm of $\gamma$ is simply the
diameter of the range of $\gamma$.  The $V^1$ norm of $\gamma$ is finite
precisely when $\gamma$ is rectifiable, and when $\gamma$ is differentiable
it corresponds exactly with the length $|\gamma|$ of $\gamma$ defined earlier.
It is easy to see the monotonicity property
$$ \| \gamma \|_{V^p} \leq \| \gamma \|_{V^r} \hbox{ whenever } 1 \leq r \leq p \leq \infty$$
and the triangle inequalities
$$
(\| \gamma_1 \|_{V^r}^r + \| \gamma_2 \|_{V^r}^r)^{1/r} \leq
\| \gamma_1 + \gamma_2 \|_{V^r} \leq
\| \gamma_1 \|_{V^r} + \| \gamma_2 \|_{V^r}$$
where $\gamma_1 + \gamma_2$ is the concatenation of $\gamma_1$ and
$\gamma_2$.
A key fact about the $V^r$ norms is that they can be subdivided:

\begin{lemma}\label{subdivide}  Let $\gamma: [a,b] \to G$ be a continuously
differentiable curve with
finite $V^r$ norm.  Then there exists a decomposition $\gamma = \gamma_1 + \gamma_2$ of the curve into two sub-curves such that
$$ \| \gamma_1 \|_{V^r}, \| \gamma_2 \|_{V^r} \leq 2^{-1/r} \| \gamma \|_{V^r}.$$
\end{lemma}

\begin{proof}
Let $t_* = \sup \{t \in [a,b] : \|\gamma|_{[a,t]}\|_{V^r} \leq  2^{-1/r} \| \gamma \|_{V^r}\}.$ Letting $\gamma_1 = \gamma|_{[a,t_*]}$ we have $\|\gamma_1\|_{V^r} =  2^{-1/r} \| \gamma \|_{V^r}$. The bound for $\gamma_2 = \gamma|_{[t_*,b]}$ follows from the left triangle inequality above.
\end{proof}

Given a continuously differentiable
curve $\gamma: [a,b] \to G$, we can define its
\emph{left trace} $\gamma_l: [a,b] \to {\mathfrak g}$ by the formula
$$ \gamma_l(t) = \int_a^t \gamma(s)^{-1} \gamma'(s)\ ds$$
Note that the trace is also a continuously differentiable curve,
but taking values now in the Lie algebra ${\mathfrak g}$ instead of $G$.
Clearly $\gamma_l$ is determined uniquely from $\gamma$.
The converse is also true after specifying the initial point $\gamma(a)$
of $\gamma$, since $\gamma$ can then be recovered by solving the ordinary differential equation
\begin{equation}\label{gagal}
 \gamma'(t) = \gamma(t) \gamma'_l(t).
\end{equation}
This equation is fundamental in the theory of eigenfunctions
of a one-dimensional Schr\"o\-dinger or Dirac operator, or equivalently
in the study of the nonlinear Fourier transform; see, for example, \cite{tao03nfa}, \cite{muscalu03ctc} for a full discussion. Basically for a fixed potential $f(t)$
and a frequency $k$, the nonlinear Fourier transform traces out a curve $\gamma(t)$ (depending on $k$)
taking values in a Lie group (e.g. $SU(1,1)$), and the corresponding left
trace is essentially the ordinary linear Fourier transform.

It is easy to see that these curves have the same length (i.e. they
have the same $V^1$ norm):
\begin{equation}\label{isometry}
 |\gamma| = |\gamma_l|.
\end{equation}
We now show that something similar is true for the $V^r$ norms provided
that $r < 2$.

\begin{lemma} \label{lglemma}  Let $1 \leq r < 2$, let $G$ be a connected
finite-dimensional Lie group, and
let $\|\cdot\|_{\mathfrak g}$ be a norm on the Lie algebra
of $G$.  Then there exist a constant
$C > 0$ depending only on these above quantities,
such that for all smooth
curves $\gamma: [a,b] \to G$, we have
\begin{equation}\label{nonlinear}
\| \gamma \|_{V^r} \leq \| \gamma_l \|_{V^r} + C \min( \| \gamma_l \|_{V^r}^2, \| \gamma_l \|_{V^r}^r)
\end{equation}
and
\begin{equation}\label{linear}
\| \gamma_l \|_{V^r} \leq \| \gamma \|_{V^r} + C \min( \| \gamma \|_{V^r}^2, \| \gamma \|_{V^r}^r ).
\end{equation}
\end{lemma}

An analogous result holds for the \emph{right trace}, $\int_a^t \gamma'(s) \gamma(s)^{-1}\ ds$, once the left-invariant norm on $T_gG$ is replaced by a right-invariant norm.

\begin{proof} We may take $r > 1$ since the claim is
already known for $r=1$ thanks to \eqref{isometry}.

It shall suffice to prove the existence of a small $\delta > 0$ such that
we have the estimate
\begin{equation}\label{nonlinear-small}
 \| \gamma \|_{V^r} = \| \gamma_l \|_{V^r} + O( \| \gamma_l \|_{V^r}^2 )
\end{equation}
whenever $\|\gamma_l \|_{V^r} \leq \delta$, and similarly
\begin{equation}\label{linear-small}
 \| \gamma_l \|_{V^r} = \| \gamma \|_{V^r} + O( \| \gamma \|_{V^r}^2 )
\end{equation}
whenever $\|\gamma\|_{V^r} \leq \delta$.  (We allow the $O()$ constants here to
depend on $r$, the Lie group $G$, and the norm structure, but not on $\delta$).
Let us now see why these estimates will prove the lemma.  Let us
begin by showing that \eqref{nonlinear-small} implies \eqref{nonlinear}.
Certainly this will be the case if $\gamma_l$ has $V^r$ norm less than $\delta$.
If instead $\gamma_l$ has $V^r$ norm larger than $\delta$, we can use
Lemma \ref{subdivide} repeatedly to partition it into
$O( \delta^{-r} \| \gamma_l \|_{V^r}^r )$ curves, all of whose $V^r$ norms are
less than $\delta$.  These curves are the left-traces of various components of
$\gamma$, and thus by \eqref{nonlinear-small} these components have
a $V^r$ norm bounded by some quantity depending on $\delta$.  Concatenating
these components together (using the triangle inequality)
we obtain the result.  A similar argument
allows one to deduce \eqref{linear} from \eqref{linear-small}.

Next, we observe that to prove the two estimates \eqref{nonlinear-small},
\eqref{linear-small} it suffices to just
prove one of the two, for instance
\eqref{nonlinear-small}, as this will also imply \eqref{linear-small}
for $\| \gamma \|_{V^r}$ sufficiently small by the usual continuity argument
(look at the set of times $t$ for which the restriction of $\gamma$ to
$[a,b]$ obeys a suitable version of \eqref{linear-small}, and use
\eqref{nonlinear-small} to show that this set is both open and closed if
$\|\gamma \|_{V^r}$ is small enough).

It remains to prove \eqref{nonlinear-small} for $\delta$ sufficiently small.
We shall in fact prove the more precise statement (note $\gamma_l(a)=0$)
\begin{equation}\label{k-guard}
 \| \log(\gamma(a)^{-1} \gamma(b)) - \gamma_l(b) \|_{\mathfrak g}
\leq K \| \gamma_l \|_{V^r}^2
\end{equation}
for some absolute constant $K > 0$ (and for $\delta$
sufficiently small), where $\log$ is the inverse of the
exponential map $\exp: {\mathfrak g} \to G$.
Note that it follows from a continuity argument as in the previous paragraph that if $\delta$ is sufficiently small then $\gamma(b)^{-1} \gamma(a)$
is sufficiently close to the identity so that the logarithm is well-defined.
Let us now see why \eqref{k-guard} implies \eqref{nonlinear-small}.
Applying the inequality to any segment $[t_j, t_{j+1}]$ in $[a,b]$ we see that
$$
 \| \log(\gamma(t_{j})^{-1} \gamma(t_{j+1})) -
(\gamma_l(t_{j+1}) - \gamma_l(t_j)) \|_{\mathfrak g}
\leq K
\| \gamma_l|_{[t_j, t_{j+1}]} \|_{V^r}^2$$
and hence with Lemma \ref{terryfix} (since $\delta$ is small)
$$ d( \gamma(t_{j+1}), \gamma(t_j) ) = \| \gamma_l(t_{j+1}) - \gamma_l(t_j)
\|_{\mathfrak g} +
O( \| \gamma_l|_{[t_j, t_{j+1}]} \|_{V^r}^2 ).$$
Estimating
$O( \| \gamma_l|_{[t_j, t_{j+1}]} \|_{V^r}^2 )$
crudely by $\| \gamma_l \|_{V^r}
O( \| \gamma_l|_{[t_j, t_{j+1}]} \|_{V^r})$ and taking the $\ell^r$ sum in
the $j$ index, we see that for any partition
$a = t_0 < \ldots < t_n = b$ we have
$$ \Big(\sum_{j=0}^{n-1} d( \gamma(t_{j+1}), \gamma(t_j) )^r\Big)^{1/r} =
\Big(\sum_{j=0}^{n-1} \| \gamma_l(t_{j+1}) - \gamma_l(t_j)
\|_{\mathfrak g}^r\Big)^{1/r} +
O( \| \gamma_l \|_{V^r}^2 ).$$
Taking suprema over all partitions we obtain the result.

It remains to prove \eqref{k-guard} for some suitably large $K$.  This
we shall do by an induction on scale (or ``Bellman function'') argument.
Let us fix the smooth curve $\gamma$.  We shall
prove the estimate for all subcurves of $\gamma$, i.e. for all
intervals $[t_1, t_2]$ in $[a,b]$, we shall prove that
\begin{equation}\label{k-guard-piecemeal}
 \| \log(\gamma(t_1)^{-1} \gamma(t_2)) - (\gamma_l(t_2) - \gamma_l(t_1)) \|_{\mathfrak g}
\leq K \| \gamma_l|_{[t_1,t_2]} \|_{V^r}^2.
\end{equation}
Let us first prove this in the case when the interval $[t_1,t_2]$ is sufficiently short, say of
length at most $\eps$ for some very small $\eps$ (depending on $\gamma$).  In that case, we perform
a Taylor expansion to obtain
\begin{equation}\label{gap}
 \gamma_l(t) = \gamma_l(t_1) + \gamma'_l(t_1) (t-t_1) + \frac{1}{2} \gamma''_l(t_1) (t-t_1)^2
+ O_\gamma((t-t_1)^3)
\end{equation}
and
\begin{equation}\label{gapl}
 \gamma'_l(t) = \gamma'_l(t_1) + \gamma''_l(t_1) (t-t_1)
+ O_\gamma((t-t_1)^2)
\end{equation}
when $t \in [t_1,t_2]$, and where the $\gamma$ subscript in $O_\gamma$ means that the constants here are allowed
to depend on $\gamma$ (more specifically, on the $C^3$ norm of $\gamma$), and the $O()$ is with respect to the $\| \|_{\mathfrak g}$
norm.  Also we remark that as $\gamma$ is assumed
smooth, $\gamma'_l(t_1)$ is bounded away from zero.  It is then an easy matter to conclude that
\begin{equation}\label{vq}
 \| \gamma_l|_{[t_1,t_2]} \|_{V^r} \geq \frac{1}{2} \| \gamma'_l(t_1) \|_{\mathfrak g} |t_2-t_1|
\end{equation}
if $\eps$ is sufficiently small depending on $\gamma$.  On the other hand, from \eqref{gagal} and \eqref{gapl}
we have
$$ \gamma'(t) = \gamma(t) (\gamma'_l(t_1) + \gamma''_l(t_1) (t-t_1)
+ O_\gamma((t-t_1)^2))$$
from which one may conclude that
$$ \gamma(t) = \gamma(t_1) \exp( \gamma'_l(t_1) (t-t_1) + \frac{1}{2} \gamma''_l(t_1) (t-t_1)^2
+ O( \| \gamma'_l(t_1)^2 \|_{\mathfrak g} |t-t_1|^2 ) + O_\gamma((t-t_1)^3) )$$
for all $t \in [t_1,t_2]$, if $\gamma$ is sufficiently small.  We rewrite this as
\begin{multline*} \log( \gamma(t_1)^{-1} \gamma(t) ) \\
=\gamma'_l(t_1) (t-t_1) + \frac{1}{2} \gamma''_l(t_1) (t-t_1)^2
+ O( \| \gamma'_l(t_1)^2 \|_{\mathfrak g} |t-t_1|^2 ) + O_\gamma((t-t_1)^3),
\end{multline*}
and then specialize to the case $t = t_2$. By \eqref{gap}, we have
$$ \log( \gamma(t_1)^{-1} \gamma(t_2) ) - (\gamma_l(t_2) - \gamma_l(t_1)) =
 O( \| \gamma'_l(t_1)^2 \|_{\mathfrak g} |t_2-t_1|^2 ) + O_\gamma((t_2-t_1)^3),$$
and hence by \eqref{vq} we have \eqref{k-guard-piecemeal} if $t_2-t_1$ is small enough
(depending on $\gamma$) and $K$ is large enough (\emph{independent} of $\gamma$).

This proves \eqref{k-guard-piecemeal} when the interval $[t_1,t_2]$ is small enough.  By \eqref{vq}, it also
proves \eqref{k-guard-piecemeal} when $\| \gamma_l|_{[t_1,t_2]} \|_{V^r}$ is sufficiently small.  To conclude
the proof of \eqref{k-guard-piecemeal} in general, we now assert the following inductive claim:
if \eqref{k-guard-piecemeal} holds whenever $\| \gamma_l|_{[t_1,t_2]} \|_{V^r} < \eps$ and some given $0 < \eps \leq \delta$,
then it also holds whenever $\| \gamma_l|_{[t_1,t_2]} \|_{V^r} < 2^{1/r} \eps$, providing that
$K$ is sufficiently large (\emph{independent} of $\eps$) and $\delta$ is sufficiently small (depending on $K$, but
\emph{independent} of $\eps$).  Iterating this we will obtain the claim
\eqref{k-guard-piecemeal} for all intervals $[t_1,t_2]$ in $[a,b]$.

It remains to prove the inductive claim.  Let $[t_1,t_2]$ be any subinterval of $[a,b]$ such that
the quantity $A = \| \gamma_l|_{[t_1,t_2]} \|_{V^r}$ is less than  $2^{1/r} \eps$.  Applying
Lemma \ref{subdivide}, we may subdivide $[t_1,t_2] = [t_1,t_*] \cup [t_*,t_2]$ such that
$$ \| \gamma_l|_{[t_1,t_*]} \|_{V^r}, \| \gamma_l|_{[t_*,t_2]} \|_{V^r} \leq 2^{-1/r} A < \eps \leq r.$$
By the inductive hypothesis, we thus have
$$
 \| \log(\gamma(t_1)^{-1} \gamma(t_*)) - (\gamma_l(t_*) - \gamma_l(t_1)) \|_{\mathfrak g}
\leq K 2^{-2/r} A^2$$
and
$$
 \| \log(\gamma(t_*)^{-1} \gamma(t_2)) - (\gamma_l(t_2) - \gamma_l(t_*)) \|_{\mathfrak g}
\leq K 2^{-2/r} A^2.$$
In particular, we have
\begin{align*}
 \| \log(\gamma(t_1)^{-1} \gamma(t_*)) \|_{\mathfrak g}
&\leq \| \gamma_l(t_*) - \gamma_l(t_1) \|_{\mathfrak g} + K 2^{-2/r} A^2 \\
&\leq \| \gamma_l|_{[t_1,t_*]} \|_{V^r} + O( K A^2 ) \\
&= O( A (1 + KA ) ) = O( A (1 + K\delta) )= O(A)
\end{align*}
if $\delta$ is sufficiently small depending on $K$.  Similarly we have
$$  \| \log(\gamma(t_*)^{-1} \gamma(t_2)) \|_{\mathfrak g} = O(A)$$
and hence by the Baker-Campbell-Hausdorff formula (if $\delta$ is sufficiently small)
$$  \| \log(\gamma(t_1)^{-1} \gamma(t_2)) - \log(\gamma(t_1)^{-1} \gamma(t_*))
- \log(\gamma(t_*)^{-1} \gamma(t_2)) \|_{\mathfrak g} = O(A^2).$$
By the triangle inequality, we thus have
$$
 \| \log(\gamma(t_1)^{-1} \gamma(t_2)) - (\gamma_l(t_2) - \gamma_l(t_1)) \|_{\mathfrak g}
\leq 2 K 2^{-2/r} A^2 + O(A^2).$$
We now use the hypothesis $r < 2$, which forces $2 \times 2^{-2/r} < 1$.
If $K$ is large enough (depending on $r$, but independently of $
\delta$, $A$, or $\eps$) we thus have
\eqref{k-guard-piecemeal}.  This closes the inductive argument.
\end{proof}

Letting $w,v$ be any elements of the Lie algebra $\mathfrak g,$ one can define a nonlinear Fourier summation operator associated to $G,w,v$ by means of the left trace
\begin{align*}
\mathcal{NC}[f](k,0)&=I \\
\frac{\partial}{\partial x} \mathcal{NC}[f](k,x) &= \mathcal{NC}[f](k,x)\left(\real(e^{-2\pi i k x} f(x))w + \imaginary(e^{-2\pi i k x} f(x)) v\right)
\end{align*}
or (giving a different operator) by the right trace
\begin{align*}
\mathcal{NC}[f](k,0) &= I \\
\frac{\partial}{\partial x} \mathcal{NC}[f](k,x) &= \left(\real(e^{-2\pi i k x} f(x))w + \imaginary(e^{-2\pi i k x} f(x)) v\right) \mathcal{NC}[f](k,x).
\end{align*}
Above, $k,x \in \rea$, $\mathcal{NC}[f]$ takes values in $G$, $I$ is the identity element of $G$, and $\real,\imaginary$ are the real and imaginary parts of a complex number.
An example of interest is given by $G = SU(1,1),$ and
$$w = \left(\begin{array}{cc} 0 & 1\\1 & 0\end{array} \right),\qquad
v = \left(\begin{array}{cc} 0 & i\\-i & 0\end{array} \right).$$

Combining Lemma \ref{lglemma} with the variational Menshov-Paley-Zygmund theorem of the previous section,
we obtain a variational version of the Christ-Kiselev theorem \cite{christ01wab}. Namely, we see that for $1 \leq p < 2$ and $r > p$
$$ \big\| \indic{|{\mathcal{NC}}[f]| \leq 1} {\mathcal{NC}}[f] \big\|_{L^{p'}_k(V^r_x)} \leq C_{p,r,G,w,v} \| f \|_{L^p(\R)}$$
and
$$ \big\| \indic{|{\mathcal{NC}}[f]| \geq 1} {\mathcal{NC}}[f]\big\|_{L^{p'/r}_k(V^r_x)}^{1/r} \leq C_{p,r,G,w,v} \| f \|_{L^p(\R)}.$$
Note that the usual logarithms are hidden in the $d$ metric we have placed on the Lie group $G$.

Extending these estimates to the case $p=2$ is an interesting and
challenging problem, even when $r=\infty$, which would
correspond to a nonlinear Carleson theorem. Lemma \ref{lglemma}
cannot be extended to any exponent $r\ge 2$. Sandy Davie and
the fifth author of this paper have an unpublished example of a
curve in the Lie group ${SU}(1,1)$ with trace in the subspace of
$\mathfrak{su}(1,1)$ of matrices vanishing on the diagonal so
that the diameter of the curve is not controlled by the
$2$-variation of the trace.

Terry Lyons' machinery \cite{lyons98de} via iterated integrals
faces an obstruction in  a potential application to a nonlinear
Carleson theorem because of the unboundedness results for the
iterated integrals shown in \cite{muscalu03cme}.

\section{An application to ergodic theory}\label{ergodicappendix}

Wiener-Wintner type theorems is an area in ergodic theory that is most
closely related to the study of Carleson's operator. In \cite{lacey08wwt},
Lacey and Terwilleger prove the following singular integral variant of the
Wiener-Wintner theorem:

\begin{theorem}\label{ltwiener}
For $1<p$, all measure preserving flows
$\{T_t:t\in \R\}$ on a probability space
$(X,\mu)$ and functions $f\in L^p(\mu)$, there is
a set $X_f\subset X$ of probability one, so that for all $x\in X_f$
we have that the limit
$$\lim_{s\to 0} \int_{s<|t|<1/s} e^{i\theta t} f(T_tx)\frac {dt}t\ \ .$$
exists for all $\theta\in \R$.
\end{theorem}

One idea to approach such convergence results is to study quantitative
estimates
in the parameter $s$ that imply convergence, as pioneered by Bourgain's paper \cite{bourgain89pet} in
similar context. We first need to pass to a mollified variant of the above
theorem:

\begin{theorem}
Let $\phi$ be a function on $\R$ in the Wiener space, i.e.
the Fourier transform $\widehat{\phi}$ is in $L^1(\R)$. For $1<p$,
all measure preserving flows
$\{T_t:t\in \R\}$ on a probability space
$(X,\mu)$ and functions $f\in L^p(\mu)$, there is
a set $X_f\subset X$ of probability one, so that for all $x\in X_f$
we have that the limits
$$\lim_{s\to \infty} \int e^{i\theta t} f(T_tx) \phi(st)\frac {dt}t\ \ ,$$
$$\lim_{s\to 0 } \int e^{i\theta t} f(T_tx)\phi(st)\frac {dt}t\ \ .$$
exist for all $\theta\in \R$.
\end{theorem}

This theorem clearly follows from an a priori estimate
$$\Big\|\sup_{\theta}\big\|\int e^{i\theta t} f(T_tx) \phi(st)\frac {dt}t
\big\|_{V^r(s)} \Big\|_{L^p(x)} \le C \|f\|_{L^p}$$
for $r>\max(2,p')$. Here we have written
$V^r(s)$ for the variation norm taken in the parameter $s$ of the
expression inside, and likewise for $L^p(x)$. The variation norm is
the strongest norm widely used in this context, while
Lacey and Terwilliger use a
weaker oscillation norm in the proof of their Theorem.

By a standard transfer method
(\cite{calderon-ergodic},
\cite{coifman-weiss-transference})
involving replacing $f$ by translates $T_yf$ and
an averaging procedure in $y$, the a priori estimate can be
deduced from an
analogous estimate on the real line
\begin{equation}\label{varaverage}
\Big\| \sup_\xi \Big\| \int
e^{ \xi it} f(x+t)\phi(st)\frac {dt}t\Big\|_{V^r(s)} \Big\|_{L^p(x)}\,\lc
\,\|f\|_{L^p} \ .
\end{equation}

The main purpose of this appendix is to show how this estimate
\eqref{varaverage} can be deduced from the main theorem of this paper
by an averaging argument. We write the $V^r(s)$ norm explicitly and expand $\phi$ into a
Fourier integral to obtain for the left hand side
of \eqref{varaverage} the expression
$$\Big\| \sup_\xi \sup_{s_0<s_1<\dots<s_K}
\Big(\sum_{k=1}^K
\Big|\int\int e^{ \xi it} f(x+t) e^{i\eta (s_k-s_{k-1})t}
\frac {dt}t \widehat{\phi}(\eta) \, d\eta \Big|^r\Big)^{1/r}\Big\|_{L^p(x)}\ .$$
Now pulling the integral in $\eta$ out of the various norms and
considering only positive $\eta$ (with the case of negative $\eta$ being
similar) and defining
$\xi_k=\xi+\eta s_k$ we obtain the upper bound
$$ \int_{\eta>0} \Big\| \sup_{\xi_0<\xi_1<\dots<\xi_K}
\Big(\sum_{k=1}^K
\Big|\int e^{ i(\xi_k-\xi_{k-1}) t} f(x+t)
\frac {dt}t\Big|^r\Big)^{1/r}\Big\|_{L^p(x)}\,|\widehat{\phi}(\eta)|\, d\eta\ .$$
Now applying the variational Carleson estimate and doing the trivial
integral in $\eta$ bounds this term by a constant times $\|f\|_{L^p}$.

\begin{remark}  To prove the Lacey-Terwilleger theorem
\ref{ltwiener} from the mollified version, one may approximate the
characteristic functions used as cutoff functions by Wiener space functions so
that the difference is small in $L^1$ norm. Then at least for $f$
in $L^\infty$ one can show convergence of the limits by an approximation
argument, even though one will not recover the full strength of the
quantitative estimate in the Wiener space setting. The result for $f$ in $L^\infty$
can then be used as a dense subclass result in other $L^p$ spaces, which can
be handled by easier maximal function estimates and further approximation
arguments.
\end{remark}
\begin{remark} The classical version of the Wiener-Wintner theorem does not invoke
singular integrals but more classical averages of the type
$$ \frac{1}{2s} \int_{|t|<s} e^{i\theta t} f(T_tx)\, {dt}\  .$$
We note that the same technique as above may be applied to these
easier averages.
\end{remark}

\end{appendices}

\end{document}